\documentclass{amsart}
\usepackage{amsmath}
\usepackage{amssymb}
\usepackage{amsthm}
\usepackage{dsfont}
\usepackage{cite}
\usepackage{lineno}
\usepackage{subfigure}
\usepackage{graphicx}
\usepackage{multirow}
\usepackage{color}
\usepackage{xspace}
\usepackage{pdfsync}
\usepackage{hyperref} 
\usepackage{algorithmicx}
\usepackage{algpseudocode}
\usepackage{algorithm}
\usepackage{mathtools}
\usepackage{enumitem}
\graphicspath{{MeshfreeFigures/}}

\DeclareMathOperator*{\argmax}{argmax}

\DeclarePairedDelimiter\floor{\lfloor}{\rfloor}
\newcommand{\bq}{\begin{equation}}
\newcommand{\eq}{\end{equation}}
\newcommand{\R}{\mathbb{R}}

\newcommand{\abs}[1]{\left\vert#1\right\vert}

\newcommand{\G}{\mathcal{G}}
\newcommand{\bO}{\mathcal{O}}

\newcommand{\Dt}{\mathcal{D}}

\newcommand{\Tf}{\mathcal{T}}
\newcommand{\torus}{\mathbb{T}}

\newcommand{\Lf}{\mathcal{L}}

\newcommand{\Nf}{\mathcal{N}}
\newcommand{\Nb}{\mathbb{N}}

\newcommand{\MA}{Monge-Amp\`ere\xspace}

\newcommand{\phiplus}{\phi^h}

\algnewcommand{\LineComment}[1]{\State \(\triangleright\) #1}

\newtheorem{theorem}{Theorem}
\theoremstyle{lemma}
\newtheorem{lemma}[theorem]{Lemma}
\newtheorem{corollary}[theorem]{Corollary}
\newtheorem{definition}[theorem]{Definition}

\newtheorem{remark}[theorem]{Remark}
\newtheorem{hypothesis}[theorem]{Hypothesis}
\theoremstyle{remark}

\makeatletter
\newcommand\appendix@section[1]{%
\refstepcounter{section}%
\orig@section*{Appendix \@Alph\c@section: #1}%
}
\let\orig@section\section
\g@addto@macro\appendix{\let\section\appendix@section}
\makeatother

\begin{document}

\title[finite difference schemes on compact manifolds]{On the reduction in accuracy of finite difference schemes on manifolds without boundary}

\author{Brittany Froese Hamfeldt}
\address{Department of Mathematical Sciences, New Jersey Institute of Technology, University Heights, Newark, NJ 07102}
\email{bdfroese@njit.edu}
\author{Axel G. R. Turnquist}
\address{Department of Mathematical Sciences, University of Texas at Austin, 2515 Speedway, Austin, TX 78712}
\email{agt6@njit.edu}

\thanks{The authors were partially supported by NSF DMS-1751996. }
\begin{abstract}
We investigate error bounds for numerical solutions of divergence structure linear elliptic PDEs on compact manifolds without boundary.  Our focus is on a class of monotone finite difference approximations, which provide a strong form of stability that guarantees the existence of a bounded solution.  In many settings including the Dirichlet problem, it is easy to show that the resulting solution error is proportional to the formal consistency error of the scheme.  We make the surprising observation that this need not be true for PDEs posed on compact manifolds without boundary.  {We propose a particular class of approximation schemes built around an underlying monotone scheme with consistency error $O(h^\alpha)$. By carefully constructing barrier functions, we prove that the solution error }is bounded by $O(h^{\alpha/(d+1)})$ in dimension $d$.  We also provide a specific example where this predicted convergence rate is observed numerically.  Using these error bounds, we further design a family of provably convergent approximations to the solution gradient.
\end{abstract}

\date{\today}    
\maketitle

In this article, we develop new convergence rates for numerical schemes for solving elliptic partial differential equations (PDEs) posed on compact manifolds without boundary.  A surprising result, which is also demonstrated empirically, is that the solution error need not be proportional to the consistency error of the scheme, even for the smoothest problems.

One fruitful approach to solving fully nonlinear elliptic PDEs such as the \MA equation has been to construct monotone discretizations. This avenue of research was inspired by~\cite{BSNum}, where the authors proved that a consistent, monotone, and stable numerical discretization is guaranteed to converge uniformly to the weak (viscosity) solution of the PDE, provided the underlying PDE satisfies a strong comparison principle. Notably, this result does not apply to many PDEs posed on manifolds without boundary, since those equations do not have a strong comparison principle. There is a large body of work on constructing monotone schemes in Euclidean space, see~\cite{benamou2014monotone, BenamouDuval, BFO_OTNum, mirebeau,chenwanlin, FroeseMeshfreeEigs, FO_MATheory, HS_Quadtree, HamfeldtBVP2,  HL_LagrangianGraphs, HL_ThreeDimensions, junliu, Nochetto_MAConverge, ObermanSINUM, ObermanEigenvalues}. The authors~\cite{xiaobingfenglewis2,xiaobingfenglewis1} introduced the notion of generalized monotonicity as an alternative approach to producing convergent methods for some nonlinear elliptic PDEs~\cite{xiaobingfenglewis3, xiaobingfenglewis4}.  

The authors of the present article have recently introduced a monotone method and convergence proof for a class of \MA type equations posed on the sphere~\cite{HT_OTonSphere,HT_OTonSphere2}.  The approximation techniques developed there extend naturally to many other elliptic PDEs posed on the sphere.  However, the theory guarantees only convergence of the methods, without providing any information about error bounds.


In this manuscript, we begin the process of developing convergent rates for numerical schemes on a compact manifold $M$ by considering linear elliptic divergence structure equations of the form
\begin{equation}\label{eq:PDE}
{-\text{div}_{M} \left( A(x) D_{M}u(x)\right)} + f(x) = 0  
\end{equation}
where $A(x)$ is symmetric positive definite.


{We denote
\bq\label{eq:PDEoperator}
\mathcal{L}[u](x) \equiv  -\text{div}_{M} \left( A(x) D_{M}u(x)\right)
\eq
and} notice immediately that the nullspace of this PDE operator consists of constants.  
{Thus solutions to the PDE~\eqref{eq:PDE} are, at best, unique only up to additive constants.}
We hereby fix any point $x_0 \in M$ and further impose the {additional} condition
\begin{equation}\label{eq:uniqueness}
u(x_0) = 0.
\end{equation}

There exist fairly general conditions upon which there exists a weak $H^1(M)$ solution to~\eqref{eq:PDE},~\eqref{eq:uniqueness} provided the given data satisfies the solvability condition
\begin{equation}\label{eq:solvability}
\int_{M} f(x)dx = 0.
\end{equation}
See~\cite[{Theorem~4.7}]{aubin}. This solvability condition arises naturally from the fact that $\mathcal{L}$ is self-adjoint and thus
\[ \int_M f(x)\,dx = - \int_M \mathcal{L}[u]\,dx = -\int_M u\mathcal{L}^*[1]\,dx = 0. \]

The linearized version of the \MA equation arising in Optimal Transport is an example of such a PDE {that has been well studied in Euclidean space}~\cite{linearization}. However, in our setting such linear elliptic PDEs will be posed on compact manifolds without boundary. Thus, they lack boundary conditions and the usual approaches of establishing convergence rates for numerical schemes do not work.

We have investigated the surprising fact that for manifolds without boundary it is possible to construct simple monotone discretizations of linear elliptic PDEs in one dimension for which the empirical convergence rate is asymptotically worse than the formal consistency error. Buttressing this, we derive explicit convergence rates on  {more general} manifolds without boundary.  In particular, we find that the error is bounded by $\mathcal{O}\left( h^{\alpha/(d+1)} \right)$ where $h^\alpha$ is the formal consistency error of the scheme, $d$ is the dimension of the manifold, and $h$ is the discretization parameter. This somewhat surprising result demonstrates even more clearly the need to design higher-order schemes for solving such elliptic PDE on manifolds without boundary. Future work will involve relating this convergence result for linear elliptic PDE in divergence form to nonlinear PDEs.

The availability of convergence rates also allows us to build new schemes for approximating solution gradients.  These are guaranteed to converge, whereas standard consistent finite difference approximations need not correctly approximate the gradient of a numerically obtained function.  We produce a family of gradient approximations, with error bounds that are limited by the $\mathcal{O}\left( h^{\alpha/(d+1)} \right)$ bounds on the $L^\infty$ solution accuracy.

In Section~\ref{sec:background}, we provide an overview of important background information relating to the numerical solution and analysis of elliptic PDEs on manifolds. In Section~\ref{sec:empirical}, we provide a simple one-dimensional example that illustrates numerically the reduction in accuracy that can occur in the absence of boundary conditions. In Section~\ref{sec:convergence}, we establish convergence rates for monotone schemes for linear uniformly elliptic PDE on manifolds without boundary. In Section~\ref{sec:mapping}, we show how these convergence rates can be used to devise convergent wider-stencil approximations of the solution gradient. {In Section~\ref{sec:results}, we provide computational results to validate the error bounds and techniques described in this article.}

\section{Background}\label{sec:background}
\subsection{{Linear} elliptic PDEs on manifolds}
The specific focus of the present article is linear divergence structure PDEs of the form
{
\bq
-\text{div}_M(A(x)\nabla_Mu(x)) + f(x) = 0,
\eq}
which are defined on a compact manifold $M$.
These equations are elliptic if $A$ is a symmetric positive definite matrix.  

Given sets $\Omega\subset\R^2$, $\Omega'\subset M$ and local coordinates $y:\Omega\to\Omega'$ we can locally recast this as the following linear divergence structure operator in Euclidean space:
\bq\label{eq:localcoords} -\frac{1}{\sqrt{\det G}}\nabla\cdot\left(\sqrt{\det G}A G^{-1} \nabla u\right) {+f(x) = 0}\eq
where $G$ is the metric tensor~\cite{cabre}.

{The results of this article are particularly motivated by the study of nonlinear generalizations of this PDE, such as Monge-Amp\`ere type equations.  The numerical solution of these nonlinear equations on manifolds is of growing interest in applications such as computer graphics~\cite{cui2019spherical},  optical design problems~\cite{anthonissen2021unified}, and data science~\cite{Peyre_ComputationalOT}.  However, little is known about error bounds or the approximation of solution gradients in this setting.  The present study of linear PDEs on manifolds will lay the groundwork for an ultimate generalization to the nonlinear setting.}


\subsection{{Monotone approximation schemes}}
%
To build approximation schemes for the PDE~\eqref{eq:PDE}, we begin with a point cloud $\G^h\subset M$ discretizing  the underlying manifold and let
\begin{equation}\label{eq:h}
h = \sup\limits_{x\in M}\min\limits_{y\in\G^h} d_{M}(x,y)
\end{equation}
denote the characteristic (geodesic) distance between discretization nodes.
In particular, this guarantees that any ball of radius $h$ on the manifold will contain at least one discretization point.

In this manuscript, we will consider finite difference discretizations of the PDE~\eqref{eq:PDE} of the form
\bq\label{eq:approx1} F^h \left( x,u(x),u(x)-u(\cdot) \right) = 0, \quad x\in \G^h. \eq

%
%

Critically, the approximation scheme~\eqref{eq:approx1} needs to be \emph{consistent} with the underlying PDE~\eqref{eq:PDE}.
\begin{definition}[Consistency error]\label{consistency}
We say that the approximation $F^h$ of the PDE operator $F$ has consistency error $\mathcal{O} \left( h^{\alpha} \right)$ if for every smooth $\phi\in C^{2,1}(M)$ there exists a constant $C$ such that
\[
 \|F^h(x,\phi(x),\phi(x)-\phi(\cdot)) - F \left(x, \nabla\phi(x), D^2 \phi(x) \right)\|_{L^\infty(\G^h)} \leq C h^{\alpha}
\]
for every sufficiently small $h>0$.
\end{definition}

\begin{remark}
{In this article, we assume conditions that ensure solutions lie in the H\"older space $C^{2,1}(M)$.}  It is also possible to design approximation schemes that depend on higher-order derivatives of the solution; indeed, this is assumption is typically needed for schemes with {superlinear} consistency error ($\alpha>{1}$).  The schemes analyzed in this article are required to satisfy an additional monotonicity assumption, which limits the consistency error to at most second-order ($\alpha \leq 2$).  See~\cite[Theorem~4]{ObermanSINUM}.
\end{remark}


Another concept that has proved important in the numerical analysis of elliptic equations is \emph{monotonicity}~\cite{BSNum}.  At its essence, monotone schemes reflect at the discrete level the elliptic structure of the underlying PDE.  This allows one to establish key properties of the discretization including a discrete comparison principle.  Even in the linear setting, monotonicity can play an important role in establishing well-posedness and stability of the approximation scheme~\eqref{eq:approx1}.

\begin{definition}[Monotonicity]\label{def:monotone}
The approximation scheme $F^h$~{\eqref{eq:approx1}}  is \emph{monotone} if it is a non-decreasing functions of its final two arguments.
\end{definition}

Closely related to monotonicity is the concept of a \emph{proper} scheme.
\begin{definition}[Proper]\label{def:proper}
The finite difference scheme $F^h$~{\eqref{eq:approx1}}  is \emph{proper} if {there exists a constant $C>0$ such that
\[ F^h(x,u,p)-F^h(x,v,p) \geq C(u-v) \]
whenever $u > v$.
}
\end{definition}

We note that any consistent, monotone scheme $F^h$ can be perturbed to a  proper scheme by defining
\[ G^h(x,u,p) = F^h(x,u,p) + \epsilon^hu \]
where $\epsilon^h\to0$ as $h\to0$. 

Monotone, proper schemes satisfy a strong form of the discrete comparison principle~\cite[Theorem~5]{ObermanSINUM}.  {Remarkably, this is the case even when the underlying PDE does not satisfy a comparison principle~\cite{HamfeldtBVP2}.}
\begin{theorem}[Discrete comparison principle]\label{thm:discreteComparison}
Let $F^h$ be a proper, monotone finite difference scheme and suppose that
\[ F^h(x,u(x),u(x)-u(\cdot)) \leq F^h(x,v(x),v(x)-v(\cdot)) \]
for every $x\in\G^h$.
Then $u \leq v$.
\end{theorem}

Finally, we make a continuity assumption on the scheme in order to guarantee the existence of a discrete solution.
\begin{definition}[Continuity]\label{def:continuous}
The scheme $F^h$~{\eqref{eq:approx1}} is \emph{continuous} if it is continuous in its final two arguments.
\end{definition}
\begin{remark}
We recall that the domain of the first argument of $F^h$ is the discrete set $\G^h$.  Thus it is not meaningful to speak about continuity with respect to the first argument.
\end{remark}

Critically, continuous, monotone, and proper schemes always admit a unique solution~\cite[Theorem~8]{ObermanSINUM}.  Moreover, under mild additional assumptions, it is easy to show that the solution can be bounded uniformly independent of $h$.
\begin{lemma}[Solution bounds]\label{lem:properBounds}
Suppose the PDE~\eqref{eq:PDE} has a $C^{2,1}$ solution.  Let $F^h$ be continuous, monotone,  proper, and have consistency error $\mathcal{O}(h^\alpha)$.  Suppose also that there exists a constant $C>0$, independent of $h$, such that for every $\delta>0$,
\[ F^h(x,u+\delta,p) \geq F^h(x,u,p) + Ch^\alpha\delta. \]
Then for every sufficiently small $h>0$, the scheme~\eqref{eq:approx1} has a unique solution $u^h$ that is uniformly bounded independent of $h$.
\end{lemma}

{
\begin{remark}
We note that because the scheme~\eqref{eq:approx1} is proper, it will admit a unique solution even if the underlying PDE~\eqref{eq:PDE} does not have a unique solution.  The bound obtained in Lemma~\ref{lem:properBounds} depends on the particular solution of~\eqref{eq:PDE} used for reference, and need not be tight.
\end{remark}
}

\begin{proof}[{Proof of Lemma~\ref{lem:properBounds}}.]
Since $F^h$ is continuous, monotone, and proper, a solution $u^h$ exists by~\cite{ObermanSINUM}.
Let $u$ be {any $C^{2,1}$} solution of~\eqref{eq:PDE}.  By consistency, we know that there exists a constant $K$, which does not depend on $h$, such that \[-Kh^{\alpha} \leq F^h(x,u(x), u(x) - u(\cdot)) \leq Kh^{\alpha}.\]

Now let $M$ be any constant and compute
\begin{align*}
F^h(x,u(x)+M, u(x) - u(\cdot)) &\geq F^h(x,u(x), u(x) - u(\cdot)) + Ch^{\alpha} M\\
  &\geq (-K+CM)h^\alpha.
\end{align*}
Thus by choosing $M > K/C$, we find that
\[
F^h(x,u(x)+M, u(x) - u(\cdot)) > 0 = F^h(x,u^h(x), u^h(x) - u^h(\cdot)).
\]
Then by the Discrete Comparison Principle {(Theorem~\ref{thm:discreteComparison})},
\[u+M \geq u^h.\]

By an identical argument, we obtain
\[ u-M \leq u^h. \]

We conclude that
\[ \|u^h\|_{L^\infty} \leq \|u\|_\infty + M \]
and thus $u^h$ is uniformly bounded.
\end{proof}

%

\section{Empirical Convergence Rates in One Dimension}\label{sec:empirical}

This section will consider the very simple example of Laplace's equation on the one-dimensional torus $\torus^1$:
\bq\label{eq:torus1D}
\begin{cases}
-u''(x) = 0, & x \in \torus^1\\
u(0) = 0,
\end{cases}
\eq
which has the trivial solution $u(x) = 0$.

We use this toy problem to demonstrate several surprising properties of consistent and monotone approximations on compact manifolds, which motivate and validate the main results presented in the remainder of this article.  In particular, we observe that: 
\begin{enumerate}
\item[1.] Consistent, monotone, proper schemes need not converge to the true solution unless the solvability condition~\eqref{eq:solvability} is carefully taken into account.   \item[2.] Typical approaches for proving convergence rates for linear elliptic PDEs with Dirichlet boundary conditions fail on compact manifolds.  \item[3.] Actual error bounds achieved by convergent schemes can be asymptotically worse than the truncation error of the finite difference approximation.  \item[4.] A simple consistent scheme for the gradient need not produce a convergent approximation of the gradient when applied to a numerically obtained solution.
\end{enumerate}

%
%

\subsection{A non-convergent scheme}
We begin by describing a natural ``textbook'' approach to attempting to solve~\eqref{eq:torus1D} numerically, which does not lead to a convergent scheme. 

Consider the uniform discretization of the one-dimensional torus
\[ x_i = i h, \quad i = 0, \ldots, n-1 \]
where $h = 1/n$. Let $L^h$ be a consistent, monotone approximation of the Laplacian and let $f^h$ be a consistent approximation of the right-hand side (which is zero in this case).  We would like to solve the discrete system
\bq\label{eq:torus1D_discrete1}
L^h(x_i,u^h(x_i),u^h(x_i)-u^h(\cdot)) = f^h(x_i), \quad i = 0, \ldots, n-1.
\eq
However, this does not enforce the additional uniqueness constraint $u^h(0) = 0$.  Adding this as an additional equation leads to an over-determined system.  Instead, a natural approach is to replace the equation~\eqref{eq:torus1D_discrete1} at $x_0 = 0$ with this additional constraint.  This leads to the system
\bq\label{eq:torus1D_discrete2}
\begin{cases}
L^h(x_i,u^h(x_i),u^h(x_i)-u^h(\cdot)) = f^h(x_i), \quad i = 1, \ldots, n-1\\
u^h(x_0) = 0.
\end{cases}
\eq

As a specific implementation, we consider a wide-stencil approximation of the Laplacian, which mimics the type of scheme that is often necessary for monotonicity in higher dimension~\cite{Kocan,MotzkinWasow}.  We also make the scheme proper, which ensures that the system~\eqref{eq:torus1D_discrete2} has a unique solution~\cite{ObermanSINUM}.

Let $n=4^k$ be a perfect square (where $k\in\Nb$).  We will build schemes with stencil width $\sqrt{n} = 2^k$.  Define
\bq\label{eq:torus1D_laplacian}
L^h(x_i,u(x_i),u(x_i)-u(\cdot)) = -\frac{u(x_{i+\sqrt{n}}) + u(x_{i-\sqrt{n}}) - 2u(x_i)}{nh^2} + h(1+x_i)u(x_i)
\eq
and
\bq
f^h(x_i) = h.
\eq

The resulting approximation~\eqref{eq:torus1D_discrete1} is consistent with~\eqref{eq:torus1D}, monotone, and proper.  The use of wide stencils degrades the truncation error of the usual centered scheme from $\bO(h^2)$  to $\bO(h)$, which is of the same order as the consistency error introduced by the proper term and the approximation of the right-hand side.  Nevertheless, the discrete solution obtained by solving the system~\eqref{eq:torus1D_discrete2} does \emph{not} converge to the true solution of~\eqref{eq:torus1D}.

\begin{figure}%
\subfigure[]{\includegraphics[width=0.45\textwidth]{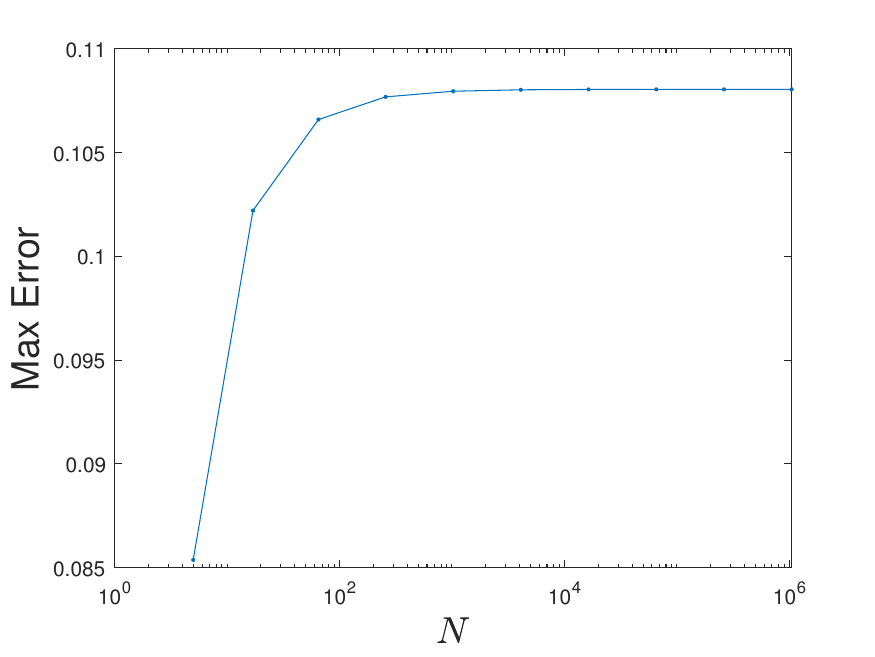}\label{fig:error1}}
\subfigure[]{\includegraphics[width=0.45\textwidth]{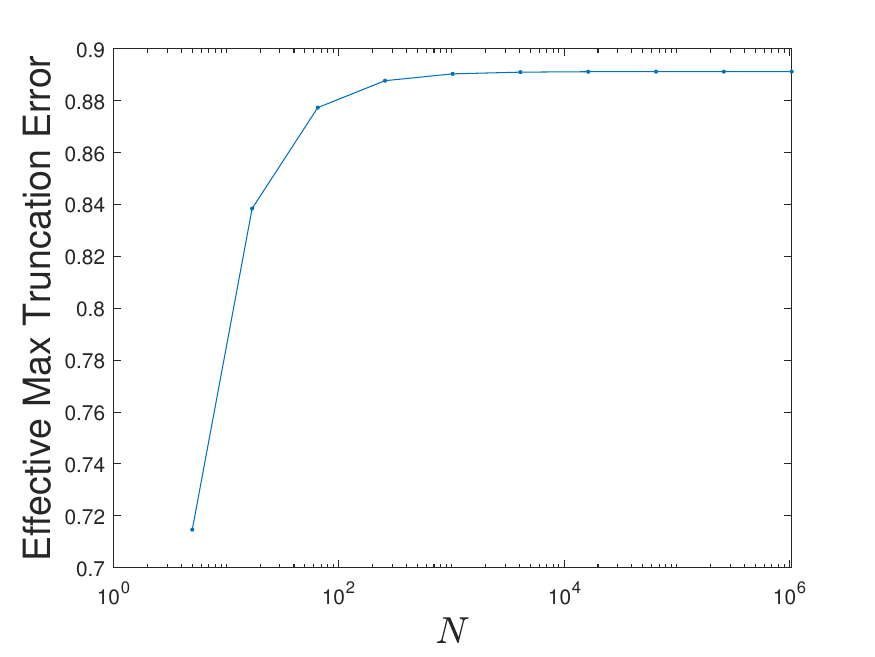}\label{fig:truncation}}
\caption{\subref{fig:error1}~Maximum error and \subref{fig:truncation}~effective maximum truncation error in the solution of~\eqref{eq:torus1D_discrete2}.}%
\label{fig:torus1D_noConvergence}%
\end{figure}

An issue that arises in this approach is that even though $L^h$ and $f^h$ are consistent with the original equation, they are not designed in a way that attempts to mimic the solvability condition~\eqref{eq:solvability} at the discrete level.  As a result, all the work of imposing this compatibility condition must be made up for at the single point $x_0=0$ where no approximation of the Laplacian is explicitly enforced in~\eqref{eq:torus1D_discrete2}.  This is evident in Figure~\ref{fig:truncation}, which plots the value of $L^h(x_0,u^h(x_0),u^h(x_0)-u^h(\cdot))$ (the ``effective'' truncation error of the scheme).  This does not converge to zero as the grid is refined.  In other words, the failure to incorporate the solvability condition at the discrete level has led to a scheme that is effectively inconsistent.

Enforcing a solvability condition at the discrete level is not straightforward: the discrete condition may not be known explicitly, and in many problems even the continuous solvability condition is not known explicitly~\cite{HL_LagrangianGraphs}.

{One solution to this challenge} is to automatically ``spread out'' the effects of the solvability condition by first solving a discrete system that is consistent with Laplace's equation at every grid point, then enforcing the uniqueness constraint in a second step.  The resulting procedure is
\bq\label{eq:torus1D_discrete3}
\begin{cases}
L^h(x_i,v^h(x_i),v^h(x_i)-v^h(\cdot)) = f^h(x_i), & i = 0, \ldots, n-1\\
u^h(x_i) = v^h(x_i) - v^h(x_0), & i = 0, \ldots, n-1.
\end{cases}
\eq

We notice that the resulting discrete solution satisfies the system
\bq\label{eq:torus1D_discrete4}
L^h(x_i,u^h(x_i),u^h(x_i)-u^h(\cdot)) = f^h(x_i) - h(1+x_i)v^h(x_0), \quad i = 0, \ldots, n-1.
\eq
This is consistent at all grid points since the first-step solution $v^h$ is uniformly bounded (Lemma~\ref{lem:properBounds}).  Moreover, the resulting solution automatically satisfies the uniqueness condition $u^h(0) = 0$ by construction.

\subsection{The Dirichlet Problem}
We are interested in establishing error bounds for solutions of~\eqref{eq:torus1D_discrete4} (and, of course, generalizations to non-trivial higher-dimensional problems).  To gain intuition and inspiration, we first review a standard approach to establishing error bounds for monotone schemes approximating the Dirichlet problem.

Consider as an example Poisson's equation with Dirichlet boundary conditions on a domain $\Omega\subset\R^d$.
\begin{equation}
\begin{cases}
- \Delta u(x) + f(x) = 0, & x \in \Omega \\
u(x) - g(x) = 0, &  x \in \partial \Omega.
\end{cases}
\end{equation}

Suppose, in addition, that we have a consistent, monotone, {proper linear} discretization scheme
\begin{equation}
\begin{cases}
L^h(x,u^h(x),u^h(x)-u^h(\cdot)) + f^h(x) = 0, & x \in \Omega\cap\G^h \\
u^h(x) - g(x) = 0, & x \in \partial \Omega \cap\G^h
\end{cases}
\end{equation}
with truncation error on the exact solution given by
\[ L^h(x,u(x),u(x)-u(\cdot)) + f^h(x) = \tau^h(x), \quad \abs{\tau^h(x)} \leq Ch^\alpha. \]

Let $z^h = u - u^h$ denote the solution error. We notice that $z^h$ satisfies the discrete system
\bq\label{eq:errorEqn}
\begin{cases}
L^h(x,z^h(x),z^h(x)-z^h(\cdot)) = \tau^h(x), & x \in \Omega\cap\G^h\\
z^h(x) = 0, & x \in \partial\Omega\cap\G^h.
\end{cases}
\eq 

If the discrete linear operator and the underlying grid are sufficiently structured, we may be able to explicitly determine its eigenvectors and eigenvalues.  In this case, we immediately obtain error bounds via
\bq\label{eq:matrixNorm} \|z^h\| \leq \|(L^h)^{-1}\| \|\tau^h\|. \eq

If the discrete problem does not have a simple enough structure, we can instead choose some bounded $w$ such that 
\[
\begin{cases}
L^h(x,w(x),w(x)-w(\cdot)) \geq 1, & x \in \Omega\cap\G^h \\
w(x) = 0, & x \in \partial \Omega \cap\G^h.
\end{cases}
\]
This can always be accomplished for a consistent approximation of a well-posed {boundary value problem}.  For example, we may choose $w$ to be the solution of the homogeneous Dirichlet problem
\bq\label{eq:wDirichlet}
\begin{cases}
-\Delta w(x) = \frac{3}{2}, & x \in \Omega\\
w(x) = 0, & x \in \partial\Omega.
\end{cases}
\eq

{
We now substitute the function $Ch^\alpha w$ into the discrete operator.  By linearity, we find that for $x\in\Omega\cap\G^h$ we have
\begin{align*}
L^h(x,Ch^\alpha w(x),Ch^\alpha w(x)-Ch^\alpha w(\cdot)) &\geq Ch^\alpha\\
  &\geq \tau^h(x)\\
  &= L^h(x,z^h(x),z^h(x)-z^h(\cdot)).
\end{align*}
Since additionally $Ch^\alpha w(x) = z^h(x) = 0$ for $x\in\partial\Omega\cap\G^h$, we can appeal to the discrete comparison principle (Theorem~\ref{thm:discreteComparison}) to conclude that
\[  z^h(x) \leq Ch^\alpha w(x), \quad x \in \G^h.\]
A similar argument yields $z^h(x) \geq -Ch^\alpha w(x)$.  Combining these, we obtain the error bound
\bq\label{eq:errorDirichlet} \|z^h\|_{L^\infty(\Omega\cap\G^h)}  \leq C\|w\|_{L^\infty(\Omega)} h^\alpha. \eq
In other words, the solution error is proportional to the truncation error of the underlying approximation scheme.
}

\subsection{Error bounds on the 1D torus}
It is natural to try to adapt the techniques used for the Dirichlet problem to error bounds for PDEs on manifolds without boundary.  Indeed, we may attempt to interpret~\eqref{eq:torus1D} as the ``one-point'' Dirichlet problem
\[
\begin{cases}
-u''(x) = 0, & x \in \torus^1 \backslash \{0\}\\
u(x) = 0, & x = 0.
\end{cases}
\]
However, this is not a well-posed PDE and attempting to solve an analog of~\eqref{eq:wDirichlet} for the auxiliary function $w$ will not lead to a function that is smooth on the torus.

We might attempt to carry this argument through at the discrete level, noticing that the solution $u^h$ of~\eqref{eq:torus1D_discrete3} does satisfy the following discrete version of a one-point Dirichlet problem
\[
\begin{cases}
L^h(x_i,u^h(x_i),u^h(x_i)-u^h(\cdot)) = f^h(x_i) - h(1+x_i)v^h(x_0), & i = 1, \ldots, n-1\\
u^h(x_0) = 0.
\end{cases}
\]
The resulting discrete linear system involves a strictly diagonally dominant $M$-matrix.  However, standard bounds on the inverse of such a matrix~\cite{ChengHuang_Mmatrix} yield the estimate
\[ \|(L^h)^{-1}\|_\infty \leq \bO\left(\frac{1}{h}\right), \]
which cannot provide any convergence guarantees when substituted into~\eqref{eq:matrixNorm}.  The degradation of this bound as $h\to0$ is due to the fact that, the the scheme~\eqref{eq:torus1D_discrete4} is proper, it is not uniformly proper as $h\to0$.

Instead, we attempt to utilize the techniques outlined above, which requires us to construct a function $w^h$ satisfying the system
\bq\label{eq:discreteW}
\begin{cases}
L^h(x_i,w^h(x_i),w^h(x_i)-w^h(\cdot)) = 1, & i = 1, \ldots, n-1\\
w^h(x_0) = 0.
\end{cases}
\eq
As this is a proper scheme, it does admit a unique solution.  However, the numerically obtained solution is not uniformly bounded as the grid is refined (Figure~\ref{fig:w}) and the resulting estimate in~\eqref{eq:errorDirichlet} does not provide a useful error bound.

\begin{figure}[htp]
	\centering
	\includegraphics[width=0.45\textwidth]{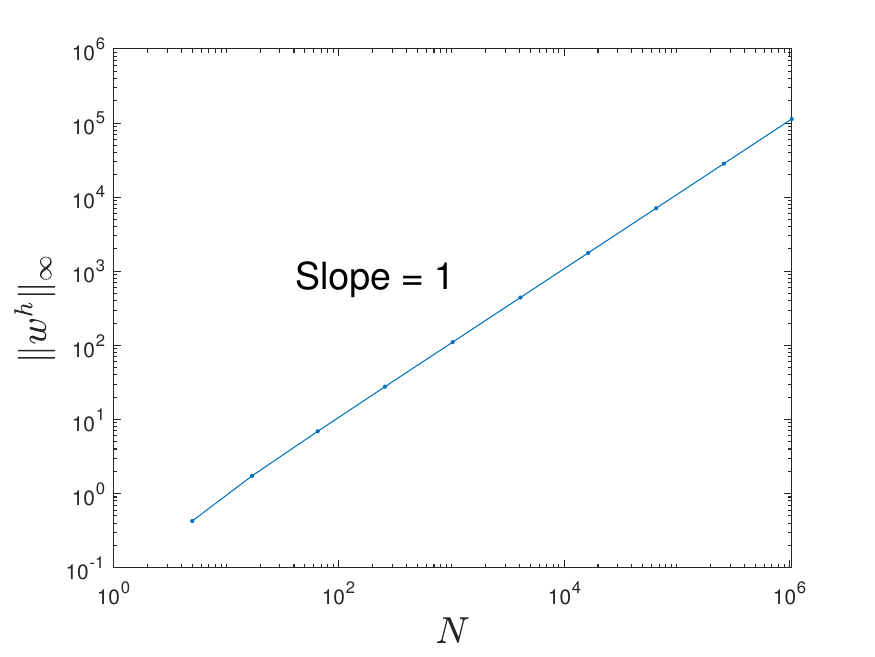}
	\caption{The maximum norm of the auxiliary function $w^h$ obtained from~\eqref{eq:discreteW}.}
	\label{fig:w}
\end{figure}

The approach we will use in \autoref{sec:convergence} to obtain error bounds involves effectively expanding the ``Dirichlet'' condition $u(x_0) = 0$ onto a larger set, which shrinks to a point as the grid is refined.  A downside to this approach is that it degrades the error bounds from the size of the truncation error $h^\alpha$ to the asymptotically worse rate of $h^{\alpha/(d+1)}$.  Surprisingly, though, our simple one-dimensional example indicates that this may be the best we can hope for.

Consider again the discrete solution $u^h$ obtained by solving~\eqref{eq:torus1D_discrete4}, which has a truncation error of $\bO(h)$ at all grid points on the one-dimensional torus and exactly satisfies the uniqueness condition $u^h(0) = 0$.  We solve this system numerically and present the error in Figure~\ref{fig:converge1}. This example does display numerical convergence to the true solution .  However, the observed accuracy is only $\bO(\sqrt{h})$, {which is asymptotically worse than the formal consistency error}.

%
%
%

\begin{figure}[htp]
\centering
	\includegraphics[width=0.45\textwidth]{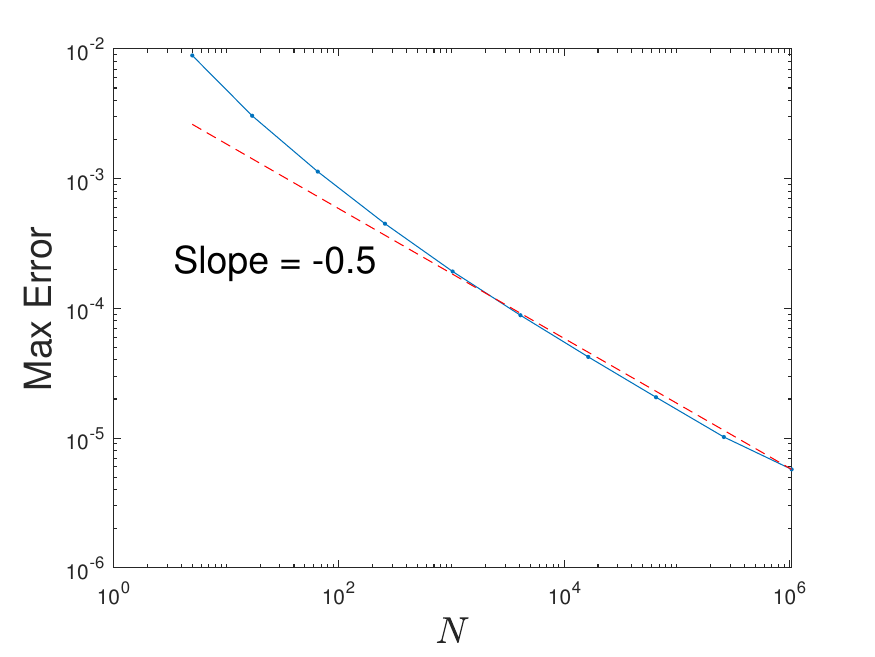}
	\caption{Maximum error in the solution of~\eqref{eq:torus1D_discrete4} on $\torus^1$.}
	\label{fig:converge1}
\end{figure}

{It is also interesting to compare the structure of the error (for fixed $n$) for the results of the non-convergent scheme~\eqref{eq:torus1D_discrete2} and the convergent scheme~\eqref{eq:torus1D_discrete4}. See Figure~\ref{fig:torus1D_errors}.  We notice that the error is approximately periodic: it is zero at the point $x_0$ (where $u(x_0)=0$ is enforced), and close to zero every $\sqrt{n}$ grid points thereafter, where the wide stencil scheme most strongly sees this condition.  At other grid points, the influence of this constraint seems to be felt more weakly. The scheme~\eqref{eq:torus1D_discrete4}, which better spreads out the effects of the solvability condition, also seems to allow this constraint to be felt more strongly so that the amplitude of the error decays as the grid is refined. 
}

{
\begin{figure}%
\centering
\subfigure[]{\includegraphics[width=0.45\textwidth]{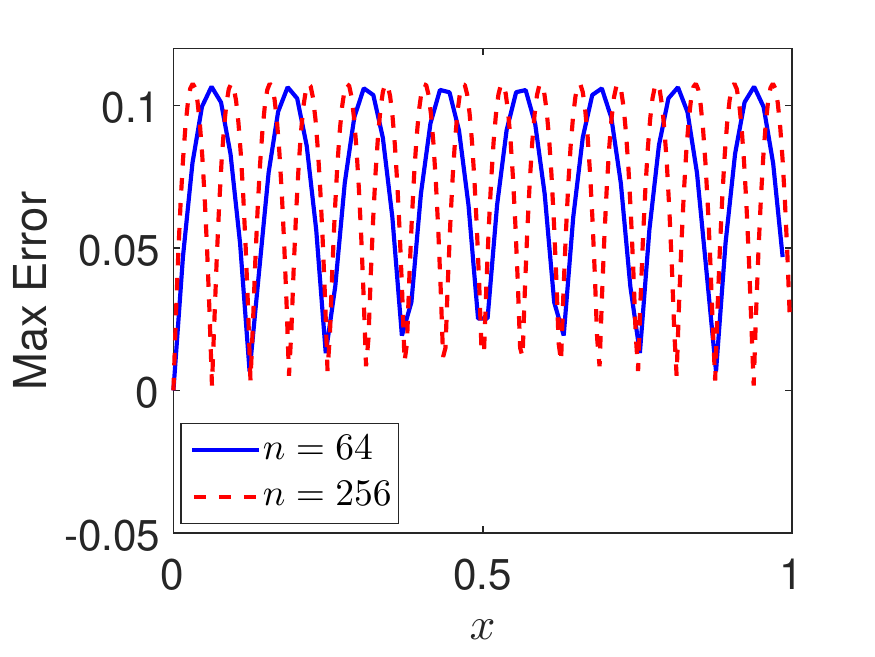}\label{fig:error_ubad}}
\subfigure[]{\includegraphics[width=0.45\textwidth]{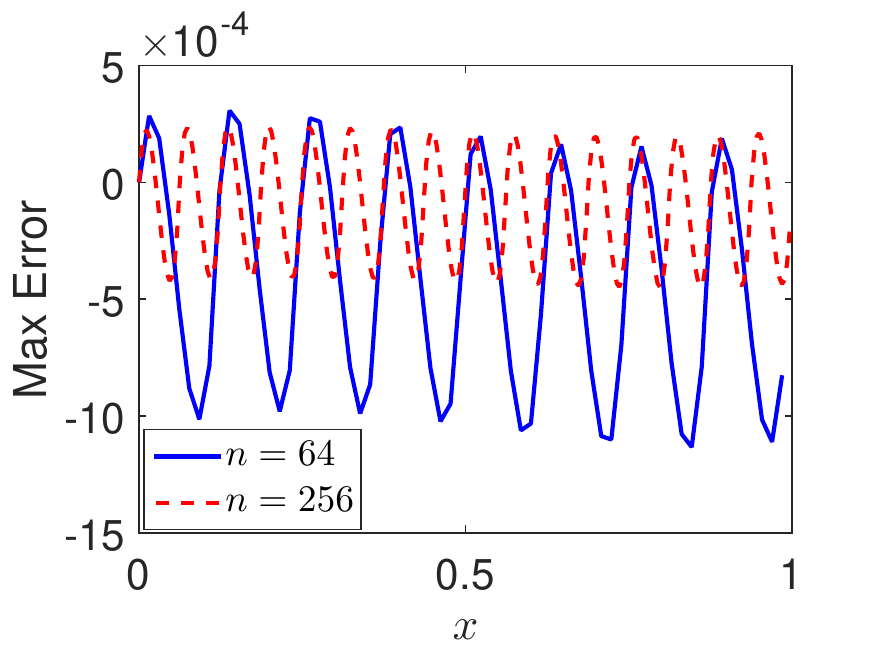}\label{fig:error_ugood}}
\caption{Error in the solutions of \subref{fig:error_ubad}~\eqref{eq:torus1D_discrete2} and \subref{fig:error_ugood}~\eqref{eq:torus1D_discrete4} for $n=64$ and $n=256$.}%
\label{fig:torus1D_errors}%
\end{figure}
}

{It appears from this simple example that computing on a manifold without boundary can lead to a reduction in the expected accuracy of a finite difference method.  This motivates us to consider in \autoref{sec:convergence} an alternate approach to ``spreading out'' the effects of the solvability condition by also ``spreading out'' the uniqueness condition $u^h(x_0)=0$ on a neighborhood of $x_0$ instead of at a single point.  The size of the neighborhood provides an immediate limit to the accuracy that can be achieved using this approach.  However, our main result (Theorem~\ref{thm:mainconvergence}) provides an error bound that is consistent with the empirical rates of convergence observed in this section for more traditional finite difference methods.}

\subsection{Convergence of gradients}
Finally, we recall the important and well-known fact in numerical analysis that pointwise convergence of an approximation does not imply convergence of gradients.  In particular, if we consider the approximation $u^h$ obtained by solving~\eqref{eq:torus1D_discrete3}, we might try to obtain information about the solution derivative by using the standard centered difference scheme
\[ u'(x_i) {=} \frac{u^h(x_{i+1})-u^h(x_{i-1})}{2h} {+ \bO(h^2)}. \]
However, this fails to converge to the true solution derivative $u'(x) = 0$ as the grid is refined; see Figure~\ref{fig:deriv1D}.

\begin{figure}%
\centering
\includegraphics[width=0.45\textwidth]{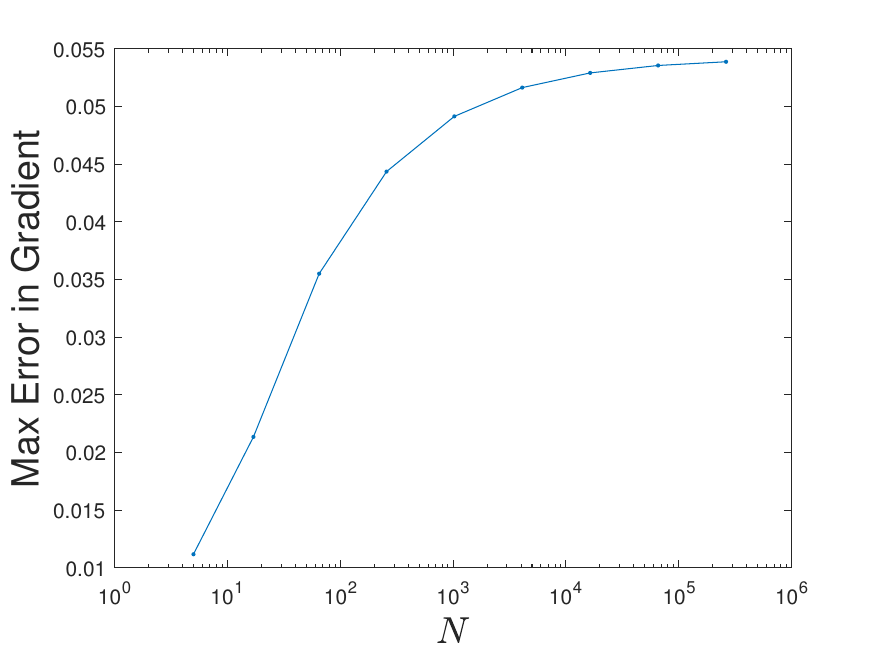}%
\caption{Maximum error in a centered difference approximation of $u'(x)$ obtained from the solution of the scheme~\eqref{eq:torus1D_discrete3}.}%
\label{fig:deriv1D}%
\end{figure}

This non-convergence is perhaps unsurprising given that $u^h$ is a low-accuracy approximation to $u$.  Indeed, a closer look at the centered difference scheme reveals that
\[ \frac{u^h(x_{i+1})-u^h(x_{i-1})}{2h} = \frac{u(x_{i+1})-u(x_{i-1}) + \bO(\sqrt{h})}{2h} {=u'(x_i) + \bO(h^2+h^{-1/2})}. \]
The theoretical error of this approximation is potentially as large as $\bO\left(h^{-1/2}\right)$, which is unbounded as $h\to0$.

Nevertheless, the numerical solution $u^h$ does still contain information about the true solution derivative.  In order to obtain this, we will require approximations of the gradient that utilize sufficiently wide stencils to overcome potential high-frequency components in the solution error.  This has the effect of making the size of the denominator in the finite difference approximation larger than the solution error, which leads to a convergent approximation as $h\to0$. This idea will be developed in \autoref{sec:mapping}.

\section{Convergence Rate Bounds}\label{sec:convergence}
We now establish error bounds for a class of consistent, monotone approximations schemes for~\eqref{eq:PDE},~\eqref{eq:uniqueness}.  The main result is presented in Theorem~\ref{thm:mainconvergence}.  The approach we take here is to construct barrier functions, which are shown to bound the error via the discrete comparison principle.  Importantly, the error estimates we obtain are consistent with the empirical convergence rates observed in~\autoref{sec:empirical}.

\subsection{Hypotheses on Geometry and PDE}

We begin with the hypotheses on the geometry $M$ and PDE~\eqref{eq:PDE} that are required by our convergence result.

\begin{hypothesis}[Conditions on PDE and manifold]
\label{hyp:convergence}
The Riemannian manifold $M$ and PDE~\eqref{eq:PDE} satisfy:
\begin{enumerate}
\item The manifold $M$ is a {$C^\infty$} compact and connected orientable {$d$-dimensional} surface without boundary.
\item The matrix $A(x)\in C^2(M)$ is symmetric positive definite.
\item The function $f(x)\in C^1(M)$ satisfies $\int_M f(x)\,dx = 0$.
\end{enumerate}
\end{hypothesis}

\begin{remark}
The compactness of the manifold $M$ implies that it is geodesically complete,
 has injectivity radius strictly bounded away from zero, and that the sectional curvature (equivalent to the Gaussian curvature in 2D) is bounded from above and below~\cite{LeeManifolds}. 
\end{remark}

\subsection{Approximation Scheme}
Next, we describe the class of approximation schemes that are covered by our convergence result.  The starting point of the scheme is the idea that the uniqueness constraint~\eqref{eq:uniqueness} should be posed at the point $x_0$, with a reasonable discrete approximation of the PDE posed on other grid points.  However, as discussed in \autoref{sec:empirical}, this approach may not yield a convergent scheme.  Instead, we will create a small cap around $x_0$ and fix the values of $u$ at all points in this cap.

To construct an appropriate scheme, we begin with any finite difference approximation $L^h(x,u(x)-u(\cdot))$ of the PDE operator~\eqref{eq:PDEoperator} that is defined for $x\in\G^h$ and that satisfies the following hypotheses.
\begin{hypothesis}[Conditions on discretization scheme]
\label{hyp:scheme}
We require the scheme $L^h$ to satisfy the following conditions:
\begin{enumerate}
\item $L^h$ is linear in its final argument.
\item $L^h$ is monotone.
\item There exist constants $C, \alpha>0$ such that for every smooth $\phi\in C^{2,1}(M)$ the consistency error is bounded by
\[
\left\vert L^h(x,\phi(x)-\phi(\cdot)) - {\Lf[\phi](x)} \right\vert \leq C[\phi]_{C^{2,1}(M)} h^{\alpha}, \quad x \in \G^h.
\]
%
\end{enumerate}
\end{hypothesis}

Next we define some regions in the manifold $M$ that will be used to create ``caps'' where $u$ is fixed in this scheme, and where additional conditions will be posed on barrier functions.  Choose any $0<\gamma < \alpha$. Define the regions
\begin{align*}
b^h &= \left\{x \in M \mid d_{M}(x,x_0) < h^{\gamma} \right\}\\
S^h &= \left\{x \in M \mid h^{\gamma} \leq d_{M}(x,x_0) \leq 2h^{\gamma} \right\}\\
B^h &= M \setminus (b^h \cup S^h).
\end{align*}
See Figure~\ref{fig:setup}.

\begin{figure}[htp]
\centering
	\includegraphics[width=0.8\textwidth]{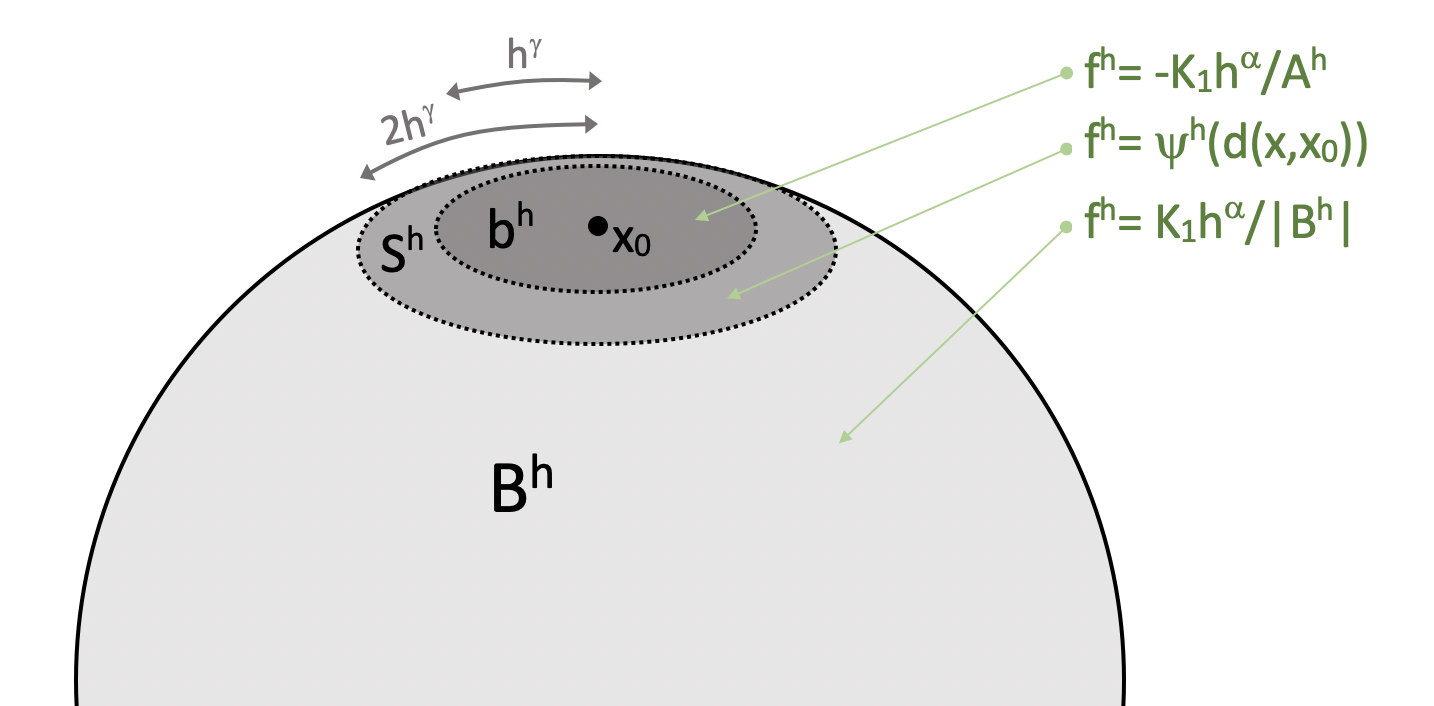}
	\caption{The construction of a small cap about $x_0$ on the manifold $M$}
	\label{fig:setup}
\end{figure}

We then define the modified scheme $F^h$ as follows:
\begin{equation}\label{eq:schemeDef}
F^h(x,u(x),u(x)-u(\cdot)) \equiv
\begin{cases}
L^h(x,u(x)-u(\cdot)) + h^{\alpha} u(x) + f(x), & x \in B^h\cap\G^h \\
u(x), & x \in (S^h\cup b^h) \cap \G^h.
\end{cases}
\end{equation}

\begin{remark}
{
Solving $F^h(x,u^h(x),u^h(x)-u^h(\cdot))=0$ has the effect of forcing $u^h(x)=0$ on the entire cap $S^h\cup b^h$.  This can 
be relaxed provided the local Lipschitz constant of $u^h$ in this cap is uniformly bounded as $h\to0$.  Pinning the value to zero has the particularly strong effect of setting the local Lipschitz constant to zero.
}
\end{remark}

Note that the discretization $F^h$ is automatically proper by construction. Therefore, this scheme has a uniformly bounded solution by Lemma~\ref{lem:properBounds}. 

\begin{lemma}\label{thm:boundedness}
Under the assumptions of Hypothesis~\ref{hyp:convergence},\ref{hyp:scheme}, the discrete scheme
\bq\label{eq:scheme} F^h(x,u^h(x),u^h(x)-u^h(\cdot)) = 0\eq
has a unique solution $u^h$ that is bounded uniformly independent of $h$ for sufficiently small $h>0$.
%
%
\end{lemma}

\subsection{Convergence Rates}
The idea in this section is to establish the convergence of the discrete solution of a monotone (and proper) scheme to the unique solution of the underlying PDE. We accomplish this by constructing a barrier function $\phi^h$ such that
{
\begin{equation}
F^h(x,-\phi^h(x),-(\phi^h(x)-\phi^h(\cdot))) \leq F^h(x,z^h(x),z^h(x)-z^h(\cdot)) \leq F^h(x,\phi^h(x),\phi^h(x)-\phi^h(\cdot))
\end{equation}
where $z^h(x) = u^h(x)-u(x)$ is the solution error.
We then by invoking the discrete comparison principle to conclude that
\begin{equation}
-\phi^h \leq z^h \leq \phi^h.
\end{equation}
}

The barrier function can be chosen to satisfy $\phi^h = \mathcal{O} \left( h^{\alpha/(d+1)} \right)$. 
In Section~\ref{sec:empirical}, we saw for $\mathbb{T}_1$ that the empirical convergence rate was $\mathcal{O} \left( h^{\alpha/2} \right)$, which is consistent with our theoretical error bound when $d=1$. The factor $(d+1)$ appears because there is a contribution of $d$ from the dimension of the underlying manifold (which arises due to the solvability condition~\eqref{eq:solvability}), and a contribution of $1$ from deriving a Lipschitz bound (also constrained by the solvability condition). Thus, we see that it is the solvability condition on the manifold without boundary that leads to the reduced convergence rate overall of a monotone and proper discretization.

We state the main convergence result:

\begin{theorem}[Convergence Rate Bounds]\label{thm:mainconvergence}
Under the assumptions of Hypotheses~\ref{hyp:convergence} and~\ref{hyp:scheme}, let $u \in C^{2,1}(M)$ be the solution of~\eqref{eq:PDE}, ~\eqref{eq:uniqueness}. Then {for sufficiently small $h>0$} the discrete solution $u^h$ solving~\eqref{eq:scheme}  {with $\gamma = \alpha/(d+1)$} satisfies
\begin{equation}
\left\Vert u^h - u \right\Vert_{L^\infty(\G^h)} \leq Ch^{\alpha / {(d+1)}}.
\end{equation}
where $C>0$ is a constant independent of $h$.
\end{theorem}







\subsubsection{Construction of barrier functions}
We now define the barrier functions $\phi^h$ by solving a linear PDE on the manifold $M$ with an appropriately chosen (small) right-hand side $f^h$ that satisfies the solvability condition~\eqref{eq:solvability}. In particular, given a fixed $K_0>0$ (which will be determined later), we let $\phi^h$ be the solution of the PDE
\begin{equation}\label{eq:phiplus}
\begin{cases}
\mathcal{L} [\phi^h](x) = f^h(x), \quad x \in M \\
\phi^h (x_0) = K_0 h^{\gamma}.
\end{cases}
\end{equation}
We emphasize that while the barrier function ${\phi}^h$ depends on the grid parameter $h$, it is the solution of the PDE on the continuous level.
%

Now we outline the construction of an appropriate function $f^h$; see Figures~\ref{fig:setup} and~\ref{fig:barrier} for two complementary visualizations of the resulting function $f^h(x)$.  Let $K_1>0$ be a fixed constant, to be determined later.  We let $\abs{U} = \int_U dx$ denote the volume of a set $U\subset M$ and note that
\[ \abs{B^h} = \bO(1), \quad \abs{S^h},\,\abs{b^h} = \bO(h^{{d}\gamma}). \]
We define the following real numbers
\begin{align*}
Q^h &= \int_{S^h} \cos\left(\pi\frac{d(x,x_0)-h^\gamma}{h^\gamma}\right)\,dx\\
A^h &= \abs{B^h}\frac{2\abs{b^h}+\abs{S^h}+Q^h}{2\abs{B^h}+\abs{S^h}-Q^h}.
\end{align*}
We record the fact that $\abs{Q^h} \leq \abs{S^h} = \bO(h^{{d}\gamma})$ and $A^h   \geq ch^{{d}\gamma}$ for some $c>0$.
Finally, we introduce a smooth cutoff function
\[ \psi^h(t) = -\frac{K_1h^\alpha}{2}\left(\frac{1}{A^h} + \frac{1}{\abs{B^h}}\right)\cos\left(\pi\frac{t-h^\gamma}{h^\gamma}\right) + \frac{K_1h^\alpha}{2}\left(\frac{1}{\abs{B^h}}- \frac{1}{A^h}\right).\]

Now we define the right-hand side function by
\bq\label{eq:fh}
f^h(x) = 
\begin{cases}
\dfrac{K_1h^\alpha}{\left\vert B^h \right\vert}, & x \in B^h \\
\psi^h(d(x,x_0)), & x \in S^h \\
-\dfrac{K_1h^\alpha}{A^h}, & x \in b^h.
\end{cases}
\eq
In particular, this is chosen to be on the order of the local truncation error of~\eqref{eq:schemeDef} throughout most of the domain, but is allowed to take on larger values in the small cap $S^h\cup b^h$ in order to ensure the solvability condition is satisfied.  See Figure~\ref{fig:barrier}.


\begin{figure}[htp]
	\includegraphics[width=\textwidth]{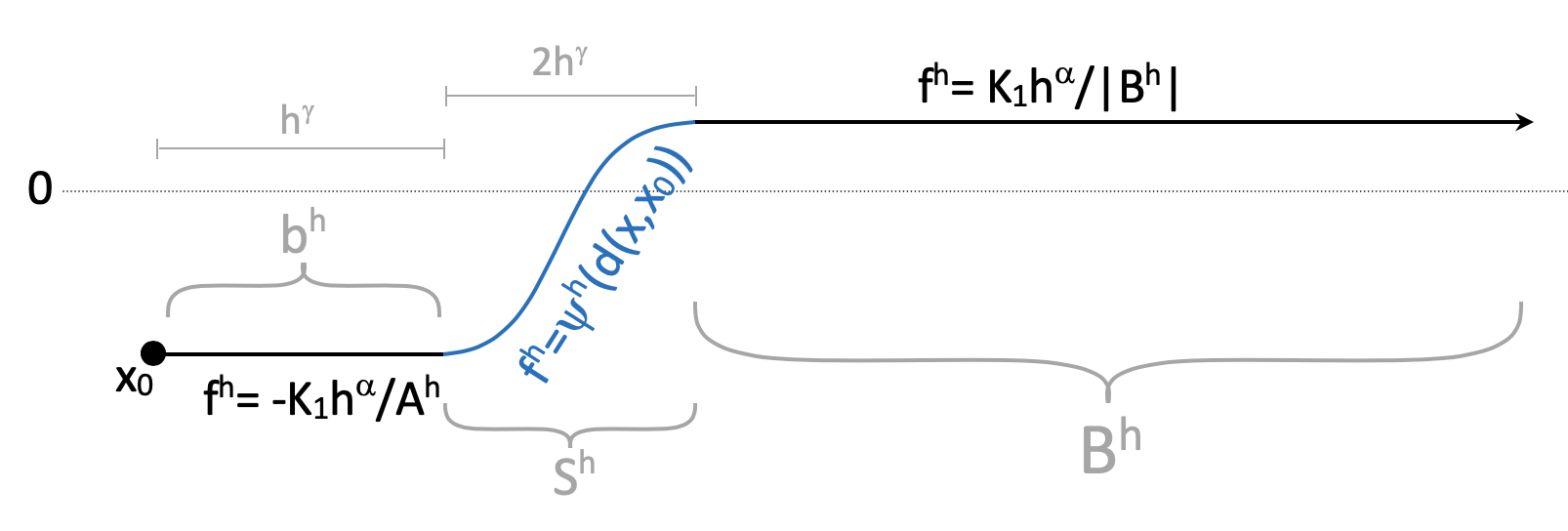}
	\caption{The construction of the function $f^h_{\pm}$ from a ``side profile" parametrized by distance from the point $x_0$}
	\label{fig:barrier}
\end{figure}

\subsubsection{Properties of the barrier function equation}
Next we verify several key properties of the right-hand side function $f^h$, which will in turn be used to produce estimates on the barrier function $\phi^h$.

\begin{lemma}[Mean-zero]\label{lem:meanZero}
For every sufficiently small $h>0$, the function $f^h$ defined in~\eqref{eq:fh} satisfies the solvability condition~\eqref{eq:solvability}
\[ \int_M f^h(x)\,dx = 0. \]
\end{lemma}

\begin{proof}
We can directly compute
\begin{align*}
\int_{M} f^h(x)dx 
 &= \int_{B^h} \frac{K_1 h^{\alpha}}{\left\vert B^h \right\vert} dx + \int_{S^h} \psi^h \left( d(x,x_0) \right)dx - \int_{b^h} \frac{K_1 h^{\alpha}}{A^h} dx\\
 &= K_1 h^{\alpha} \left( 1 - \frac{Q^h}{2} \left(\frac{1}{A^h} + \frac{1}{\left\vert B^h \right\vert}  \right) + \frac{\left\vert S^h \right\vert}{2} \left(\frac{1}{\left\vert B^h \right\vert} - \frac{1}{A^h} \right)- \frac{\left\vert b^h \right\vert}{A^h} \right)\\
&=\frac{K_1 h^{\alpha}}{A^h \left\vert B^h \right\vert} \left( A^h \left\vert B^h \right\vert - \frac{Q^h}{2}  \left(\left\vert B^h \right\vert + A^h\right) + \frac{\abs{S^h}}{2}  \left( A^h -
 \left\vert B^h \right\vert\right) - \left\vert b^h \right\vert \left\vert B^h \right\vert \right)\\
&= \frac{K_1 h^{\alpha}}{A^h \left\vert B^h \right\vert} \left(\frac{A^h}{2}\left(2\abs{B^h}+\abs{S^h}-Q^h\right) - \frac{\abs{B^h}}{2}\left(2\abs{b^h}+\abs{S^h}+Q^h\right)\right).
\end{align*}

Then by substituting in the value of $A^h$, we obtain
\[\int_{M} f^h(x)dx = 0.
\]
\end{proof}

\begin{lemma}[Regularity of right-hand side]\label{lem:fhC1}
For every sufficiently small $h>0$, $f^h \in C^1(M)$. {Moreover, $\|\nabla_M f^h \|_{L^\infty(M)} = \bO(h^{\alpha-(d+1)\gamma})$.}
\end{lemma}
\begin{proof}
First we recall that $f^h$ is constant in the regions $b^h$ and $B^h$ respectively.  In the region $S^h$, we can easily verify that
\[ \lim\limits_{d(x,x_0) \downarrow h^{\gamma}} \psi^h \left( d(x,x_0) \right) = -\frac{K_1 h^{\alpha}}{A^h}, \quad \lim\limits_{d(x,x_0) \uparrow 2h^{\gamma}} \psi^h \left( d(x,x_0) \right) = \frac{K_1 h^{\alpha}}{\left\vert B^h \right\vert},\]
which coincide with the values in $b^h$ and $B^h$ respectively.

Next, we note that
\[
\frac{d}{dt} \psi^{h} (t) = \frac{\pi K_1 h^{\alpha}}{2 h^{\gamma}} \left( \frac{1}{A^{h}} + \frac{1}{\left\vert B^{h} \right\vert} \right) \sin \left( \pi \frac{t - h^{\gamma}}{h^{\gamma}} \right).
\]
Thus we readily verify that \[\lim\limits_{t \downarrow h^{\gamma}} \frac{d}{dt}\psi^{h}(t) = 0, \quad\lim\limits_{t \uparrow 2h^{\gamma}} \frac{d}{dt}\psi^{h}(t) = 0.\] 

Finally, we produce an explicit Lipschitz bound.
\begin{align*}
\left\vert \nabla_{M} \psi^h (d(x,x_0)) \right\vert  
 &\leq \max\limits_t \left\vert \frac{d}{dt} \psi^h(t) \right\vert \\
&= \frac{\pi K_1 h^{\alpha}}{2 h^{\gamma}} \left( \frac{1}{A^{h}} + \frac{1}{\left\vert B^{h} \right\vert} \right).
\end{align*}
Using our previous observations about the size of $A^h$ and $\abs{B^h}$, we conclude that
\[ \left\vert \nabla_{M} f^h(x) \right\vert  \leq \frac{\pi K_1 h^{\alpha}}{2 h^{\gamma}} \left( \frac{1}{ch^{{d}\gamma}} + \frac{1}{\left\vert B^{h} \right\vert} \right) = \bO(h^{\alpha-{(d+1)}\gamma}).\]
\end{proof}

\begin{lemma}[$L^{{d}}$ norm bounds]\label{lem:fhL2}
There exists a constant $C>0$ such that for every sufficiently small $h>0$, 
\[ \|f^h\|_{L^{{d}}(M)} \leq Ch^{\alpha-{(d-1)}\gamma}. \]

\begin{proof}
We can directly compute
\begin{align*}
\left\Vert f^h \right\Vert_{L^{{d}}(M)} &\leq \left( \int_{S^h \cup b^h} \left( \frac{K_1 h^{\alpha}}{A^h} \right)^{{d}}dx +\int_{B^h} \left( \frac{K_1 h^{\alpha}}{\left\vert B^h \right\vert} \right)^{{d}} dx \right)^{1/{{d}}}\\
&= K_1 h^{\alpha} \left( \frac{\left\vert S^h \cup b^h \right\vert}{(A^h)^{{d}}}  + {\abs{B^h}^{1-d}} \right)^{1/{{d}}}\\
  &\leq Ch^{\alpha}\left(\frac{h^{{{d}}\gamma}}{h^{{{d^2}}\gamma}} +{\abs{B^h}^{1-d}}\right)^{1/{{d}}}.
\end{align*}

Here we have used the fact that $\left\vert S^h \cup b^h \right\vert = \bO(h^{{{d}}\gamma})$, $\abs{B^h}=\bO(1)$, and $A^h \geq c h^{{{d}}\gamma}$ for some constant $c>0$.  We conclude that
\[
\left\Vert f^h \right\Vert_{L^{{d}}(M)} {= \bO( h^{\alpha - {(d-1)}\gamma})}.
\]
\end{proof}
\end{lemma}

Using these properties of $f^h$, we are now able to establish existence of the barrier functions $\phi^h$.

\begin{lemma}[Existence of barrier function]\label{lem:phiExists}
There exists a function $\phi^h\in C^3(M)$ satisfying~\eqref{eq:phiplus}.
\end{lemma}
\begin{proof}
Recall that $f^h\in C^1(M)$ for any $h>0$.  Then by~\cite[Theorem 4.7]{aubin} we have the existence of a solution $\phi^h\in C^3(M)$ to the PDE \[\mathcal{L} \phi^h(x) = f^h(x),\] which is unique up to an additive constant. The condition $\phi^h(x_0) = K_0h^{\gamma}$ fixes the constant.
\end{proof}

\subsubsection{Local coordinate patches}
Our goal is to use regularity results for linearly elliptic PDEs in Euclidean space in order to develop estimates for the barrier function $\phi^h$, which solves a linearly elliptic PDE on the manifold $M$.  In order to do this, we will need the ability to locally re-express the barrier equation~\eqref{eq:phiplus} as a uniformly elliptic PDE in Euclidean space.  

\begin{lemma}[PDE on local coordinate patches]\label{lem:localPDE}
Under the assumptions of Hypothesis~\ref{hyp:convergence}, there exists some $r>0$ such that for every $x_0\in M$ there exists a bounded region $\Omega\subset\R^{{d}}$ and set of coordinates $y:\Omega\to B(x_0,r)$ corresponding to a metric tensor $G\in C^2(M)$ such that the PDE operator~\eqref{eq:PDEoperator} can be expressed as
\[ \mathcal{L}[\phi] = -\nabla\cdot\left((\det A)^{1/2} \nabla\phi\right). \]
{Here $B(x_0,r)\subset M$ denotes an open ball on the manifold.}
\end{lemma}
\begin{proof}
Let $x_0 \in M$ and fix any $r<r_I$ where $r_I$ is the injectivity radius of the manifold $M$.  Then we can consider a bounded set $\Omega\subset\R^2$ and a set of coordinates $y:\Omega \to B(x_0,r)$. In local coordinates~\cite{cabre1997nondivergent}, the PDE operator~\eqref{eq:PDEoperator} takes the form
\[
\mathcal{L}[\phi] = \frac{-1}{\sqrt{\det G}}\nabla \cdot\left(\sqrt{\det G} AG^{-1}\nabla \phi\right), \quad y \in \Omega.
\]

Now we choose a local metric such that $G = (\det A)^{-1/2}A$. We note that $G\in C^2(M)$ is strictly positive definite since $A$ has both these properties.  We note that $\det(G) = 1$ so that the PDE in local coordinates becomes
\[ \mathcal{L}[\phi] = -\nabla\cdot\left((\det A)^{1/2} \nabla\phi\right). \]
This is a uniformly elliptic operator since $A$ is positive definite.
\end{proof}

Importantly, because our manifold is compact, we can cover it with finitely many coordinate patches.
\begin{lemma}[Finite covering of the manifold]\label{lem:geodesicballs}
For every $r>0$, there exists a finite set of geodesic balls $\left\{ B_{r}^{i} \right\}_{i=1}^n$ such that
\[ M \subseteq \bigcup\limits_{i=1}^n B_r^i. \]
\end{lemma}

\subsubsection{Properties of barrier function}
We can now use standard regularity results for uniformly elliptic PDEs in Euclidean space to deduce key properties of the barrier function $\phi^h$.
\begin{lemma}[Bounds on barrier function]\label{lem:phiBounded}
There exists a constant $C>0$ such that for all sufficiently small $h>0$
\[ \|\phi^h\|_{L^\infty(M)} \leq C(h^\gamma + h^{\alpha-{(d-1)}\gamma}). \]
\end{lemma}
\begin{proof}
Since $\phi^h$ is continuous on a compact manifold, it achieves a maximum and minimum at some points $x, y\in M$.  Then since $M$ is connected, we can use Lemmas~\ref{lem:localPDE}-\ref{lem:geodesicballs} to construct a finite set of balls of radius $r/4$: $\left\{B_{r/4}^i\right\}_{i=1}^n$ such that
\[ x\in B_{r/4}^n, \quad y \in B_{r/4}^1, \quad B_{r/4}^i \cap B_{r/4}^{i+1} \neq \emptyset. \]
On each corresponding (larger) ball $B_r^i$ of radius $r$, we can interpret the barrier equation~\eqref{eq:phiplus} as a uniformly elliptic divergence structure PDE on a local coordinate patch in $\R^{{d}}$.

Now we denote
\begin{equation}
\bar{\phi}^h(x) \equiv \phi^h(x) - \min_{M} \phi^h.
\end{equation}
This is non-negative, which allows us to apply the de Giorgi-Nash-Moser Harnack inequality, which applies to PDEs in divergence form.
Taking $q={2d}$ in~\cite[Theorems~8.17-8.18]{gilbargtrudinger}, there exists a constant $C>0$ such that for every $i=1, \ldots, n$ we have
\[ \sup_{B_{r/4}^i} \bar{\phi}^h \leq C \left( \inf_{B_{r/4}^i} \bar{\phi}^h + \left\Vert f^h \right\Vert_{L^{{d}}(M)} \right).
 \]

Recalling that $\bar{\phi}^h(y) = 0$, we find that
\[ \sup_{B_{r/4}^1} \bar{\phi}^h \leq C_1 \|f^h\|_{L^{{d}}(M)}.\]
Now we use this to obtain an estimate in the ball $B_{r/4}^2$, which overlaps with $B_{r/4}^1$.  
\begin{align*} \sup_{B_{r/4}^2} \bar{\phi}^h 
&\leq C_1 \left( \inf_{B_{r/4}^2} \bar{\phi}^h + \left\Vert f^h \right\Vert_{L^{{d}}(M)} \right) \\
&\leq C_1 \left( \sup_{B_{r/4}^1} \bar{\phi}^h + \left\Vert f^h \right\Vert_{L^{{d}}(M)} \right) \\
&\leq C_2\|f^h\|_{L^{{d}}(M)}.
\end{align*}

Continuing this chaining argument $n$ times, we find that
\[ \|\bar{\phi}^h\|_{L^\infty(M)} = \bar{\phi}^h(x) \leq C_n\|f^h\|_{L^{{d}}(M)}. \]

By Lemma~\ref{lem:fhL2}, $\|f^h\|_{L^{{d}}(M)} = \bO(h^{\alpha-{(d-1)}\gamma})$.  We recall also that
\[ \min\limits_M \phi^h \leq \phi^h(x_0) = K_0h^\gamma, \]
which completes the proof.
\end{proof}

\begin{lemma}[Derivative bounds]\label{lem:phideriv}
There exists a constant $C>0$ such that for all sufficiently small $h>0$
\[ \left\Vert \phi^h \right\Vert_{C^{1} (M)} +\left\Vert \phi^h \right\Vert_{C^{2} (M)} + \left[ \phi^h \right]_{C^{2,1} (M)} \leq C (h^\gamma+h^{\alpha-{(d+1)}\gamma})
. \]
\end{lemma}
\begin{proof}
As in the previous lemma, we can use Lemmas~\ref{lem:localPDE}-\ref{lem:geodesicballs} to construct a finite set of balls of radius $r/2$: $\left\{B_{r/2}^i\right\}_{i=1}^n$ such that
on each corresponding (larger) ball $B_r^i$ of radius $r$, we can interpret the barrier equation~\eqref{eq:phiplus} as a uniformly elliptic divergence structure PDE on a local coordinate patch in $\R^{{d}}$.

We now  apply a classical interior regularity result for uniformly elliptic PDE~\cite[Corollary~6.3]{gilbargtrudinger}.  In particular, there exists a constant $C>0$ such that for every $i=1, \ldots, n$ we have
\[
\left\Vert \phi^h \right\Vert_{C^{1} (B_{r/2}^i)} + \left\Vert \phi^h \right\Vert_{C^{2} (B_{r/2}^i)} + \left[ \phi^h \right]_{C^{2,1} (B_{r/2}^i)} \leq C \left( \left\Vert \phi^h \right\Vert_{L^{\infty}(B_r^i)} + \left\Vert f^h \right\Vert_{C^{0, 1}(B_r^i)} \right).
\]

Then a corresponding H\"{o}lder estimate over the entire manifold is obtained by summing the estimates over the $n$ coordinate patches.
Thus we find that
\[
\left\Vert \phi^h \right\Vert_{C^{1} (M)} + \left\Vert \phi^h \right\Vert_{C^{2} (M)} + \left[ \phi^h \right]_{C^{2,1} (M)} \leq C' \left( \left\Vert \phi^h \right\Vert_{L^{\infty}(M)} + \left\Vert f^h \right\Vert_{C^{0, 1}(M)} \right).
\]

We recall from Lemmas~\ref{lem:fhC1} and~\ref{lem:phiBounded} the estimates
\[ \left\Vert f^h \right\Vert_{C^{0, 1}(M)} = \bO(h^{\alpha-{(d+1)}\gamma}), \quad \left\Vert \phi^h \right\Vert_{L^{\infty}(M)} = \bO(h^\gamma+h^{\alpha-{(d-1)}\gamma}), \]
which completes the proof.
\end{proof}

\subsubsection{Convergence rates}
The preceding regularity results allow us to select a value for $\gamma$ (which determines the radius of the small cap about $x_0$) that ensures that the family of barrier functions $\phi^h$ are uniformly Lipschitz continuous.
\begin{corollary}[Lipschitz bounds]\label{lem:lipschitz}
Let $\gamma \leq \alpha/{(d+1)}$.  Then  there exists a constant $K_\phi >0$ such that for all sufficiently small $h>0$, 
\[ \abs{\nabla\phi^h}_{C^0(M)} \leq K_\phi. \]
\end{corollary}

The requirement of Corollary~\ref{lem:lipschitz}, combined with the fact that the barrier function scales like $h^{\gamma} + h^{\alpha-{(d-1)}\gamma}$ (Lemma~\ref{lem:phiBounded}), suggests $\gamma=\alpha/{(d+1)}$ as an optimal choice.  

Now we prove the main result.

\begin{proof}[Proof of Theorem \ref{thm:mainconvergence}]
 We substitute both the error $u^h-u$ and the barrier $\phi^h$ into the scheme~\eqref{eq:schemeDef} at all $x\in\G^h$.

{\bf Case 1}: Let $x\in B^h\cap\G^h$. Then we can use the linearity of the scheme to compute
\begin{align*}
F^h(x,&u^h(x)-u(x),u^h(x)-u(x)-u^h(\cdot)+u(\cdot)) \\ &= \left(L^h(u^h(x)-u^h(\cdot)) + h^\alpha u^h(x) + f(x)\right)-\left(L^h(u(x)-u(\cdot)) + h^\alpha u(x)\right) \\
  &\leq -L(x,\nabla u(x), D^2u(x)) + C\left[u(x)\right]_{C^{2,1}(M)}h^\alpha + h^\alpha\|u\|_{L^\infty(M)}\\
	&= f(x) + C_1h^\alpha.
\end{align*}
Above, we have used the fact that $u$ solves the linear PDE~\eqref{eq:PDE} and $u^h$ solves the scheme~\eqref{eq:scheme}.

Similarly, we can compute
\begin{align*}
F^h(x,\phi^h(x),\phi^h(x)-\phi^h(\cdot)) &= L^h(\phi^h(x)-\phi^h(\cdot)) + h^\alpha\phiplus + f(x)\\
  &\geq f^h(x) - [\phiplus]_{C^{2,1}(M)}h^\alpha - \|\phiplus\|_{L^\infty(M)}h^\alpha + f(x)\\
	&\geq \frac{K_1h^\alpha}{\abs{B^h}}-C_2h^\alpha(h^\gamma + h^{\alpha-{(d+1)}\gamma}+h^{\alpha-{(d-1)}\gamma})+f(x)
\end{align*}
where we utilize the regularity bounds in Lemmas~\ref{lem:phiBounded}-\ref{lem:phideriv}.  Making the particular choice of $\gamma = \alpha/{(d+1)}$ yields
\[ F^h(x,\phi^h(x),\phi^h(x)-\phi^h(\cdot)) \geq h^\alpha\left(\frac{K_1}{\abs{B^h}}-C_2\right)-C_3h^{{\alpha(d+2)/(d+1)}} + f(x). \]

Then if we make the choice $K_1 > ({C_1+C_2+C_3})\abs{B_h}$ when we define the barrier functions in~\eqref{eq:phiplus}, we find that
\[ F^h(x,u^h(x)-u(x),u^h(x)-u(x)-u^h(\cdot)+u(\cdot)) < F^h(x,\phi^h(x),\phi^h(x)-\phi^h(\cdot)) \]
for sufficiently small $h>0$.

{\bf Case 2}: Let $x\in (b^h\cup S^h)\cap\G^h$.
Recalling that $u(x_0) = 0$, we can bound $u$ in this region by
\[ \abs{u(x)} \leq K_u d_M(x,x_0) \leq 2K_uh^\gamma. \]
Since $u^h(x) = 0$ uniformly in this small cap, we have
\[ F^h(x,u^h(x)-u(x),u^h(x)-u(x)-u^h(\cdot)+u(\cdot)) = u^h(x)-u(x) \leq 2K_uh^\gamma. \]
Similarly, we recall that $\phiplus(x_0) = K_0h^\gamma$ so that
\[ F^h(x,\phi^h(x),\phi^h(x)-\phi^h(\cdot)) = \phiplus(x) \geq (K_0-2K_\phi)h^\gamma. \]

Then if we make the choice $K_0 > 2(K_u+K_\phi)$ in the definition of the barrier function~\eqref{eq:phiplus}, we find that
\[ F^h(x,u^h(x)-u(x),u^h(x)-u(x)-u^h(\cdot)+u(\cdot)) < F^h(x,\phi^h(x),\phi^h(x)-\phi^h(\cdot)) \]
for sufficiently small $h>0$.

Combining these two cases, we find that
\[ F^h(x,u^h(x)-u(x),u^h(x)-u(x)-u^h(\cdot)+u(\cdot)) < F^h(x,\phi^h(x),\phi^h(x)-\phi^h(\cdot)) \]
for all $x\in\G^h$ and sufficiently small $h>0$.  This allows us to appeal to the Discrete Comparison Principle (Theorem~\ref{thm:discreteComparison}) to conclude that
\[ u^h(x)-u(x) \leq \phi^h(x), \quad x \in \G^h. \]
Combined with the maximum bound on $\phi^h$ (Lemma~\ref{lem:phiBounded}) applied to the case $\gamma=\alpha/{(d+1)}$, we obtain the result
\[ u^h(x)-u(x) \leq Ch^{\alpha/{(d+1)}}. \]

We can do the same procedure using $u(x)-u^h(x)$ to obtain the final result
\[ \|u^h-u\|_{L^\infty(\G^h)} \leq Ch^{\alpha/{(d+1)}}. \]
\end{proof}

\section{Approximation of Solution Gradients}\label{sec:mapping}
Having established convergence rates for the solution $u^h$ of the discrete operator, we can now use them to establish a convergence approximate of the gradient of $u^h$ with rates.

Given a function $u\in C^1(M)$ and its values on a discrete set of points $\G^h$, the design of approximations to its first derivatives is a well-studied ``textbook'' problem.  However, as discussed in \autoref{sec:empirical}, consistent approximations for first derivatives may not produce correct results when they are applied to a discrete approximation $u^h$ instead of the limiting function $u$.

Here we describe a framework for producing a family of convergent approximations of the gradient, which are based on a given discrete approximation $u^h$ with error bounds.  We provide error bounds for the gradient of $u^h$; unsurprisingly, these are bounded by the $L^\infty$ error in the approximation of $u^h$.  Combined with the convergence rate bounds of Theorem~\ref{thm:mainconvergence}, this immediately provides a provably convergent method for approximating the gradient of the solution to a divergence-structure linear elliptic PDE~\eqref{eq:PDE} on a compact manifold.

Let $x_0\in M$ be any point on the manifold and let $\nu\in\Tf_{x_{0}}$ be a unit vector in the tangent plane.  We focus on the construction of a discrete approximation to $\frac{\partial u(x_0)}{\partial \nu}$, the first directional derivative of $u$ in the direction $\nu$.  By projecting into the tangent plane as described in \autoref{sec:background}, this is equivalent to constructing convergent approximations of a first directional derivative in $\R^d$.

We will consider finite difference approximations of the form
\bq\label{eq:fd1}
\Dt_\nu u(x_0) = \frac{1}{r}\sum\limits_{i=1}^k a_i (u(x_i)-u(x_0))
\eq
where $x_i\in\G^h$ are discretization points satisfying $\abs{x_i-x_0} = \bO(r)$ and $r \geq h$ denotes the stencil width of this approximation.  

We make the following assumptions on the discrete solution $u^h$ and the gradient approximation~\eqref{eq:fd1}.
\begin{hypothesis}[Conditions on gradient approximation]
\label{hyp:gradient}
We make the following assumptions on the approximations:
\begin{enumerate}
\item There exists $p>0$ such that at every point $x\in\G^h$, the discrete approximation $u^h$ satisfies
\[ u(x) = u^h(x) + \bO(h^p). \]
\item The stencil width {satisfies} $r \geq h$ for every $h>0$.
\item There exist constants $C_1, C_2>0$ such that for every $\G^h$ there exist points $x_1, \ldots, x_k \in \G^h$ satisfying 
\[ C_1 r \leq \abs{x_i-x_0} \leq C_2 r, \quad i = 1, \ldots, k.\]
\item There exists $\beta>0$ such that the gradient approximation applied to the limiting function $u$ satisfies
\[ \Dt_\nu u(x_0) = \frac{\partial u(x_0)}{\partial\nu} + \bO(r^\beta). \]
\item The coefficients in the gradient approximation satisfy $a_i = \bO(1)$ as $h\to 0$.
\end{enumerate}
\end{hypothesis}

Under these assumptions, we can immediately provide error bounds for the gradient approximation applied to the discrete solution $u^h$.  Moreover, we can use these bounds to determine an optimal stencil width $r$ as a function of $h$.

\begin{theorem}[Error bounds for gradient]\label{thm:errorGrad}
Under the assumptions of Hypothesis~\ref{hyp:gradient} and choose $r = \bO\left(h^{p/(\beta+1)}\right)$.  Then 
\[ \frac{\partial u(x_0)}{\partial\nu} = \Dt_\nu u^h(x_0) + \bO\left(h^{\frac{p\beta}{\beta+1}}\right). \]
\end{theorem}

\begin{corollary}[Error bounds for solution gradient]\label{cor:errorGrad}
Assume the conditions of Hypotheses~\ref{hyp:convergence},~\ref{hyp:scheme},~\ref{hyp:gradient} are satisfied. Let $u$ be the solution of the PDE~\eqref{eq:PDE}, $u^h$ be the solution of the approximation scheme~\eqref{eq:schemeDef}, and $r = \bO\left(h^{\frac{\alpha}{(d+1)(\beta+1)}}\right)$.  Then 
\[ \frac{\partial u(x_0)}{\partial\nu} = \Dt_\nu u^h(x_0) + \bO\left(h^{\frac{\alpha\beta}{(d+1)(\beta+1)}}\right). \]
\end{corollary}

\begin{proof}[Proof of Theorem~\ref{thm:errorGrad}]
We can substitute directly into the approximation scheme to compute
\begin{align*}
\frac{\partial u(x_0)}{\partial\nu} &= \Dt_\nu u(x_0) + \bO(r^\beta)\\
  &= \frac{1}{r}\sum\limits_{i=1}^k a_i (u(x_i)-u(x_0)) + \bO(r^\beta)\\
	&= \Dt_\nu u^h(x_0) + \frac{1}{r} \bO(h^p) + \bO(r^\beta)\\
	&= \Dt_\nu u^h(x_0) + \bO\left(\frac{h^p}{h^{p/(\beta+1)}} + h^{\frac{p\beta}{\beta+1}}\right)\\
	&= \Dt_\nu u^h(x_0) + \bO\left(h^{\frac{p\beta}{\beta+1}}\right).
\end{align*}
\end{proof}

{
This result provides a means for correctly approximating solution gradients even when the accuracy of the approximate solution is very low.  Moreover, we notice that there is no requirement that the scheme used to approximate solution gradients be monotone and therefore limited to first-order accuracy ($\beta \leq 1$).  This opens up the possibility of using arbitrarily high-order gradient approximations, coupled to a carefully chosen stencil width $r$.  It is worth noting that as we take higher-order approximations ($\beta\to\infty$), we find that the best error bound we can achieve approaches $\bO\left(h^{\alpha/(d+1)}\right)$.  That is, the best possible error bound for approximations of the solution gradient is actually the same as the error bound guaranteed by Theorem~\ref{thm:mainconvergence} for the approximate solution itself.
}

{\section{Computational Results}\label{sec:results}
Finally, we provide some computational results to verify the error estimates developed in sections~\ref{sec:convergence}-\ref{sec:mapping}.

Consider a point cloud $\G^h\subset M$ that discretizes the manifold $M$.  In order to utilize the approximation scheme (and resulting error bounds) in~\eqref{eq:schemeDef}, we need to design a monotone finite difference approximation of the form $L^h(x,u(x)-u(\cdot))$ that is defined for $x\in\G^h$ and is consistent with the linear PDE operator $\Lf[u](x)$.

A variety of approaches are available for discretizing PDEs on manifolds~\cite{demlow2009higher,dziuk2013finite,fortunato2022high,lai2019solving,oneill2018second}.  Particularly simple are methods that allow the surface PDE to be approximated using schemes designed for PDEs in Euclidean space~\cite{ClosestPoint,TsaiManifolds}.  We test the error bounds using the tangent plane approach of~\cite{HT_OTonSphere2}, which can easily be used to design monotone approximation schemes. 

We will briefly summarize this scheme, before providing computational results in one and two dimensions.  We emphasize that our goal here is not to design an optimal scheme, but rather to test the predictions of Theorem~\ref{thm:mainconvergence} and Corollary~\ref{cor:errorGrad}, which can be applied to any monotone discretization scheme.  We also verify that, even using a very low order scheme,  it is possible to recover a convergent approximation to the solution gradient using the approach of section~\ref{sec:mapping}.
}

{\subsection{Geodesic normal coordinates}\label{sec:normalcoords}
We begin by recasting the equation using a convenient choice of local coordinates.
Consider the PDE 
\begin{equation}\label{eq:PDEGen}
-\text{div}_M(A(x)\nabla_Mu(x)) {+f(x) = 0}, \quad x\in M
\end{equation}
at a particular point $x_0\in M$.  We relate this to an equivalent PDE posed on the local tangent plane $\Tf_{x_0}$ through a careful choice of local coordinates.  In general, local coordinates will introduce distortions to the differential operators.  However, this problem was avoided in~\cite{HT_OTonSphere, HT_OTonSphere2} with the use of \emph{geodesic normal coordinates}, which preserve distance from the reference point $x_0$. In these coordinates the metric tensor is an identity matrix and the Christoffel symbols vanish at the point $x_0$.

Given some neighbourhood $N_{x_0}\subset M$ of the point $x_0\in M$, we let $v_{x_0}:N_{x_0}\to\Tf_{x_0}$ denote geodesic normal coordinates.  Because they are chosen to preserve distances from the point $x_0$, they satisfy
\[ d_M(x,x_0) = \| v_{x_0}(x) - x_0 \| \]
where $d_M$ represents the geodesic distance along $M$ and $\|\cdot\|$ the usual Euclidean distance on the tangent plane.

We can now introduce a local projection of $u$ onto the relevant tangent plane $\Tf_{x_{0}}$ in a neighborhood of $x_0$ as follows
\bq\label{eq:tangentFunction}
\tilde{u}_{x_0}(z) = u\left( v_{x_0}^{-1}(z) \right).
\eq
This allows us to re-express the PDE~\eqref{eq:PDEGen} at the point $x_0$ as an equivalent PDE on the local tangent plane.  We define
\bq\label{eq:Tangent}
 -\nabla \cdot \left( A(z) \nabla \tilde{u}_{x_{0}}(z) \right) {+ f(z) = 0}, \quad z \in \mathcal{T}_{x_{0}}
\eq
where now $\nabla$ is the usual Euclidean differential operator.   Because the particular choice of coordinates does not introduce distortions, the PDE operator will preserve its original form at the reference point $x_0$. In particular,
\[ \left.-\text{div}_M(A(x)\nabla_Mu(x))\right|_{x=x_0} = \left. -\nabla \cdot \left( A(z) \nabla \tilde{u}_{x_{0}}(z) \right)\right|_{z=x_0}. \]

The problem of approximating the PDE operator~\eqref{eq:PDEGen} at a point $x_0\in M$ is now reduced to the problem of approximating the operator~\eqref{eq:Tangent} at the point $x_0$ in the local tangent plane.  This  allows one to make use of any existing method for designing monotone approximation of PDEs in Euclidean space.
}

{
\subsection{Discretization}\label{sec:disc}
Now we consider any fixed discretization point $x_i\in\G^h$
and establish a computational neighborhood $N(i)$ about this point by defining
\begin{equation}\label{eq:neighborhood}
N(i) = \{j \mid x_j\in\G, \, d_{M}(x_i, x_j) \leq \sqrt{h}  \}.
\end{equation}
We seek an expression of the form $L^h(x_i,u(x_i)-u(x_j))$ that depends on the value of $u$ at nearby discretization points $x_j$ with $j\in N(i)$.

To accomplish this, we let $\Tf_{x_i}$ denote the tangent plane to $M$ at $x_i$.
Once the computational neighborhood $N(i)$ is established, the points $x_j \in N(i)$ are projected on to the local tangent plane $\mathcal{T}_{x_{i}}$ via a geodesic normal coordinate projection, 
\[z_j = v_{x_i}(x_j). \] 
We denote the resulting point cloud on the tangent plane by
\[ \Nf(x_i) = \{z_j \mid j \in N(i)\} \subset \Tf_{x_{i}}. \]
We extend the grid function $u:\G^h\to\R$ onto this point cloud on the tangent plane by identifying
\bq\label{eq:extend} u(z_j) = u(x_j), \quad z_j \in \Nf(x_i). \eq

A discretization of the PDE operator~\eqref{eq:PDE} at the point $x_i\in\G^h$ can now be obtained by designing a discretization of the tangent plane PDE operator
\[ -\nabla \cdot (A(z)\nabla u(z)), \quad z \in \Tf_{x_i} \]
at the point $x_i$.  The discretization should depend upon the values of $u$ on the tangent plane point cloud $\Nf(x_i)$, which can be related back to values of $u$ at points on the original manifold via~\eqref{eq:extend}.
}
{
\subsection{Monotone approximation schemes}\label{sec:monotone}
Recent work on generalized or meshfree finite difference methods demonstrate how monotone finite difference methods can be designed for unstructured grids in Euclidean space~\cite{demkowicz1984some, FroeseMeshfreeEigs, Nochetto_MAConverge, Seibold}.  
We briefly review the procedure for designing monotone generalized finite difference methods for approximating linear divergence structure operators of the form
\bq\label{eq:linearEuclidean} -\nabla\cdot (A(z)\nabla u(z))\eq
in Euclidean space, referring to the aforementioned works for further details.

Consider the problem of approximating~\eqref{eq:linearEuclidean} at a point $z_i$.  It is natural to want to use values of $u$ at the ``nearest neighbors'' to accomplish this.  Surprisingly, though, given any fixed stencil width, it is always possible to find a linear elliptic PDE operator that does not admit a consistent, monotone discretization on that stencil~\cite{Kocan,MotzkinWasow}.  For general degenerate elliptic operators, it is sometimes necessary to allow the stencil to grow wider as the grid is refined in order to achieve both consistency and monotonicity.

We notice that the PDE operator can be written in the form
\bq\label{eq:opform} \nabla\cdot(A\nabla u) = -\sum\limits_{k \in K} \frac{\partial}{\partial z_{k_1}}\left(a_{k}\frac{\partial u}{\partial z_{k_2}}\right) \eq
where
\[ K = \{k\in\mathbb{N}^2 \mid \|k\|_\infty \leq d\}. \]
This motivates us to seek a finite difference approximation of the form
\bq\label{eq:fdform} -\nabla\cdot(A\nabla u)(z_i) \approx -\sum\limits_{k\in K}\sum\limits_{z \in \Nf(z_i)}\sum\limits_{y\in\Nf(z_i)}c_{k}(z,y) a_k(y)u(z).\eq

The monotonicity condition requires that the coefficient of $u(z)$ be non-positive for each $z\neq z_i$.  This leads to the set of linear inequality constraints
\[ \sum\limits_{k\in K}\sum\limits_{y\in\Nf(z_i)}c_k(z,y)a_k(y) \geq 0, \quad z \in \Nf(z_i) \backslash \{z_i\}. \]
To achieve consistency, we Taylor expand the terms in~\eqref{eq:fdform} about the reference point $z_i$.  We then compare the coefficients of each term with the desired operator~\eqref{eq:opform}, which leads to a system of linear equations that must be satisfied by the coefficients $c_k(z,y)$.

In typical implementations, one possibility is to exploit the structure of the underlying PDE to set many of the coefficients $c_k(x,y)$ to zero \emph{a priori} and obtain closed form expressions for the (small) number of non-zero coefficients~\cite{FroeseMeshfreeEigs}.  Another option is to use simple analytical and computational optimization tools to numerically determine the values of the coefficients and establish bounds needed to ensure consistency~\cite{HL_ThreeDimensions,Seibold}.
}

{
\subsection{Computational Examples}
\subsubsection{One dimension}
We first return to the study of Poisson's equation on the one-dimensional torus $\mathbb{T}^1$:
\bq\label{eq:poisson1d}
\begin{cases}
-u''(x) = f(x), & x\in\mathbb{T}^1\\
u(0) = 0.
\end{cases}
\eq

We use the same $\bO(h)$ discretization studied in section~\ref{sec:empirical}, but modified to enforce the condition $u(x)=0$ in a small cap with a radius of $\bO(h^{1/2})$, as suggested by Theorem~\ref{thm:mainconvergence}.  In particular, we let $n=4^k$ be a perfect square and use the scheme
\bq\label{eq:scheme1d}
\begin{cases}
-\dfrac{u^h(x_{i+\sqrt{n}})+u^h(x_{i-\sqrt{n}})-2u^h(x_i)}{nh^2}+h(1+x_i)u^h(x_i) = f(x_i) + h, & d_{\mathbb{T}^1}(x_i,0) > 2h^{1/2}\\
u^h(x_i) = 0, & d_{\mathbb{T}^1}(x_i,0) \leq 2h^{1/2}.
\end{cases}
\eq

%
We begin by studying Laplace's equation ($f(x)=0$).
We plot the results in Figure~\ref{fig:torus1D_laplace}.  Surprisingly, we observe $\bO(h)$ error rather than the expected $\bO(h^{1/2})$. This can be explained by the fact that enforcing $u^h(x)=0$ on a small cap does not introduce any error since the exact solution is $u(x)=0$.  We also note (Figure~\ref{fig:torus1D_u0}) that this approach leads to a smoother error than the approaches studied in section~\ref{sec:empirical}.

\begin{figure}%
\centering
\subfigure[]{\includegraphics[width=0.45\textwidth]{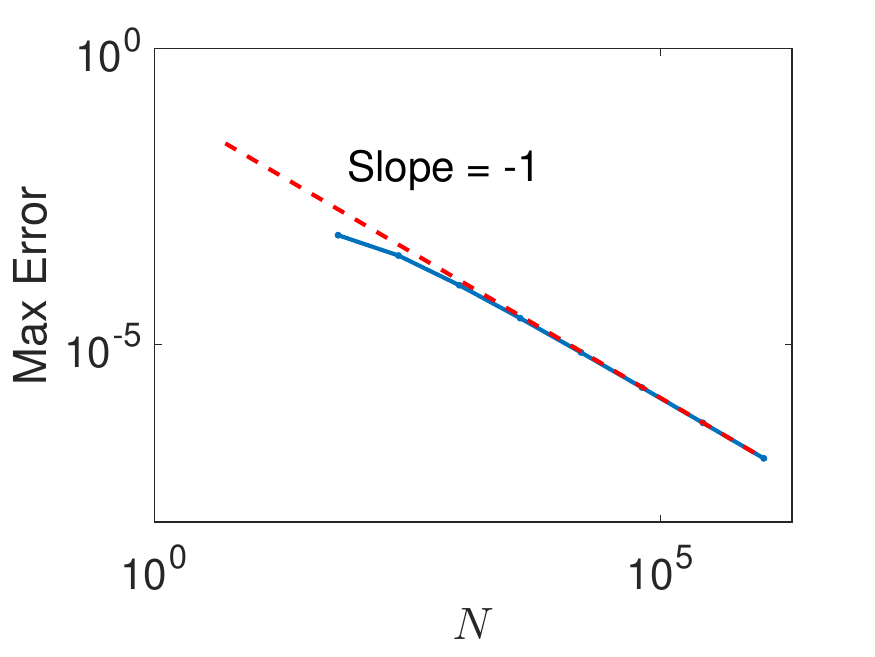}\label{fig:torus1D_u0_error}}
\subfigure[]{\includegraphics[width=0.45\textwidth]{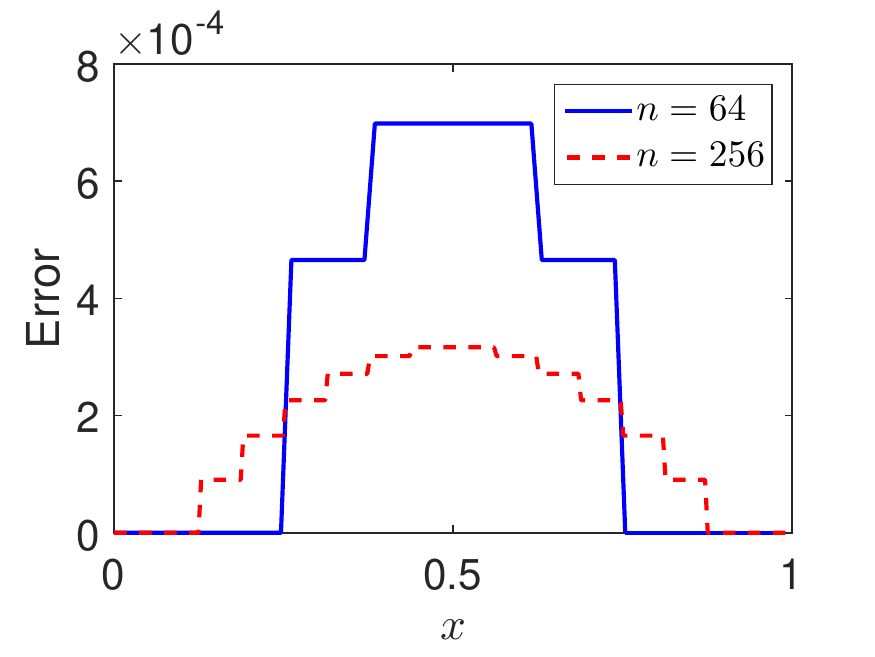}\label{fig:torus1D_u0}}
\caption{\subref{fig:torus1D_u0_error}~Maximum error and \subref{fig:torus1D_u0}~solution error for $n=64$ and $n=256$ for the solution of Laplace's equation on $\mathbb{T}^1$.}%
\label{fig:torus1D_laplace}%
\end{figure}

We secondly consider an example of a non-trivial solution ($u(x) = \sin(2\pi x)$) obtained by solving~\eqref{eq:poisson1d} with a right-hand side $f(x) = 4\pi^2\sin(2\pi x)$.  For this non-trivial example, we do observe the expected convergence rate of $\bO(h^{1/2})$ (Figure~\ref{fig:torus1D_sin_error}).

We also perform a study of the wider-stencil gradient approximations proposed in section~\ref{sec:mapping}.  We utilize three-different finite difference approximations of the derivative: a first-order forward difference ($\beta=1$), a second-order centered difference ($\beta=2$), and the following fourth order difference ($\beta=4$):
\[ \Dt_xu(x_0) = \frac{1}{r}\left(\frac{1}{12}u(x_0-2r)-\frac{2}{3}u(x_0-r)+\frac{2}{3}u(x_0+r)-\frac{1}{12}u(x_0+2r)\right).\]
In each case, we let $r = \bO(h^{\frac{1}{2(\beta+1)}})$, as suggested by Corollary~\ref{cor:errorGrad}.  The maximum error in the gradient approximations, which is plotted in Figure~\ref{fig:torus1D_deriverr}, agrees very well with the predicted $\bO(h^{\frac{\beta}{2(\beta+1)}})$ convergence rate.  In particular, we observe that as $\beta$ grows, the convergence rate becomes closer to the optimal rate of $\bO(h^{1/2})$.

\begin{figure}%
\centering
\subfigure[]{\includegraphics[width=0.45\textwidth]{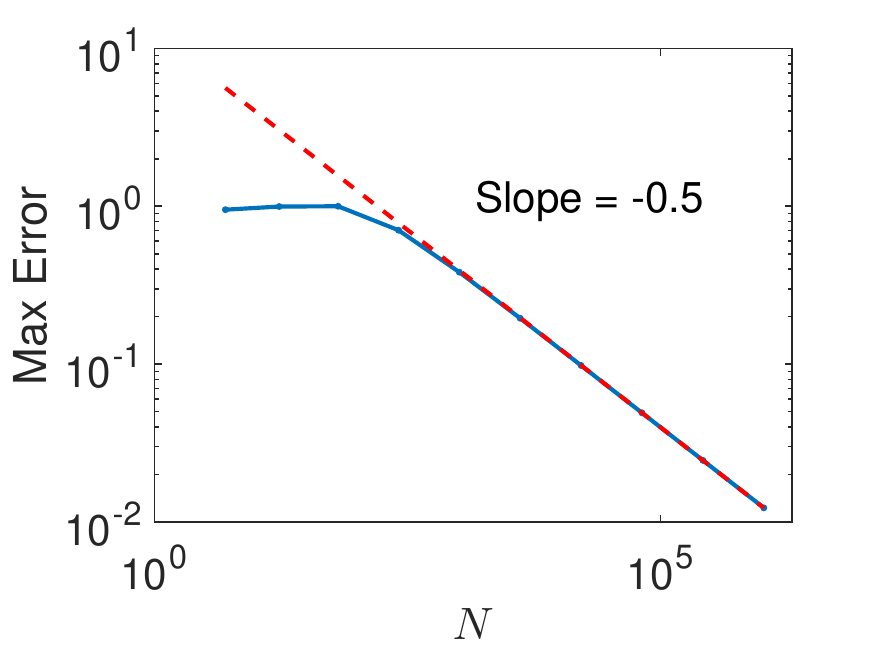}\label{fig:torus1D_sin_error}}
\subfigure[]{\includegraphics[width=0.45\textwidth]{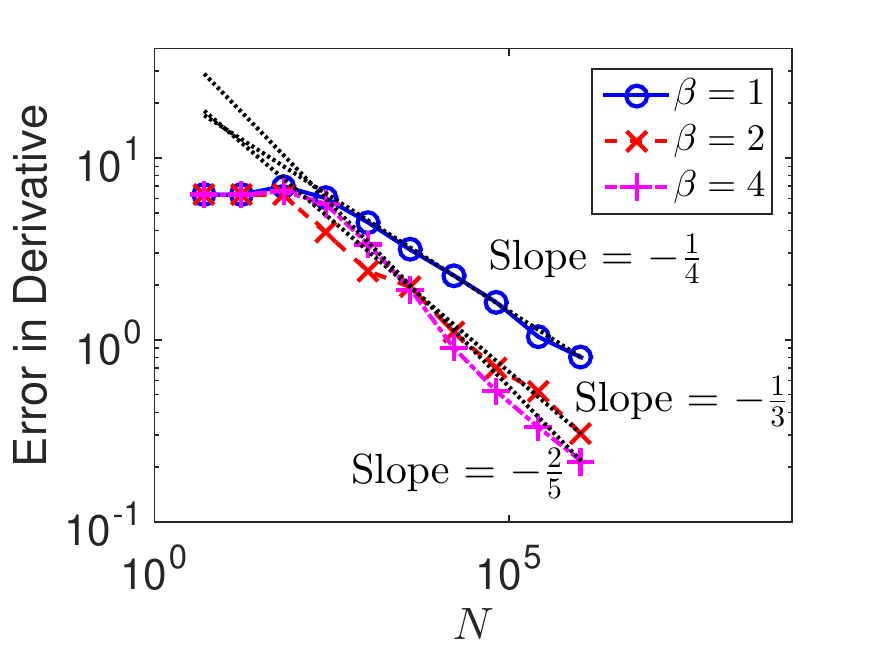}\label{fig:torus1D_deriverr}}
\caption{Maximum error in the \subref{fig:torus1D_sin_error}~solution and\subref{fig:torus1D_deriverr}~derivative for the solution of Poisson's equation on $\mathbb{T}^1$.}%
\label{fig:torus1D_poisson}%
\end{figure}

\subsubsection{Two dimensions}
Next, we demonstrate the predicted convergence rates on a two-dimensional surface using the tangent plane approach described in this section.  In light of our need to compare with an exact solution and surface gradient, we perform this test using Poisson's equation on the sphere $\mathbb{S}^2$:
\bq\label{eq:poisson2d}
\begin{cases}
-\Delta_{\mathbb{S}^2}u(\mathbf{x}) = f(\mathbf{x}), & \mathbf{x}\in\mathbb{S}^2\\
u(\mathbf{x}_0) = 0.
\end{cases}
\eq
Here we let $\mathbf{x}_0 = (0,1,0)$ and choose the right-hand side, expressed in spherical coordinates, to be
\bq\label{eq:rhs}
f(\phi, \theta) = \cos (\phi) \sin \theta \left( 1-9 \cos^2 \theta + 3 \sin^2 \theta \right).
\eq

The solution of the surface PDE~\eqref{eq:poisson2d} is
\begin{equation}
u(\phi, \theta) = \cos(\phi) \sin^3(\theta).
\end{equation}
and the surface gradient of $u$ is
\begin{equation}
\nabla_{\mathbb{S}^2} u(\phi, \theta) = -\sin(\phi)\sin(\theta) \hat{\phi} + 3\cos(\phi)\sin^2(\theta)\cos(\theta) \hat{\theta}.
\end{equation}


We discretize the sphere using a point cloud $\G^h$ consisting of $N$ points that are approximately equally spaced so that $h = \bO(\sqrt{N})$.  In particular, given some $\epsilon=\bO(h)$ and $n = \bO(h^{-1})$, we construct a layered point cloud consisting of the points
\bq
\G^h = \{(\phi_{ij},\theta_i) \mid 1 \leq i \leq n, 1 \leq j \leq \floor{n\sin\theta_i}\}.
\eq
Here we take
\[
\theta_i = \epsilon + i\frac{\pi-2\epsilon}{n}, \quad
\phi_{ij} = i\frac{1+\sqrt{5}}{2} + j\frac{2\pi}{\floor{n\sin\theta_i}}.
\]
See Figure~\ref{fig:cloud} for an example of this point cloud.

\begin{figure}%
\centering
\subfigure[]{\includegraphics[width=0.45\textwidth]{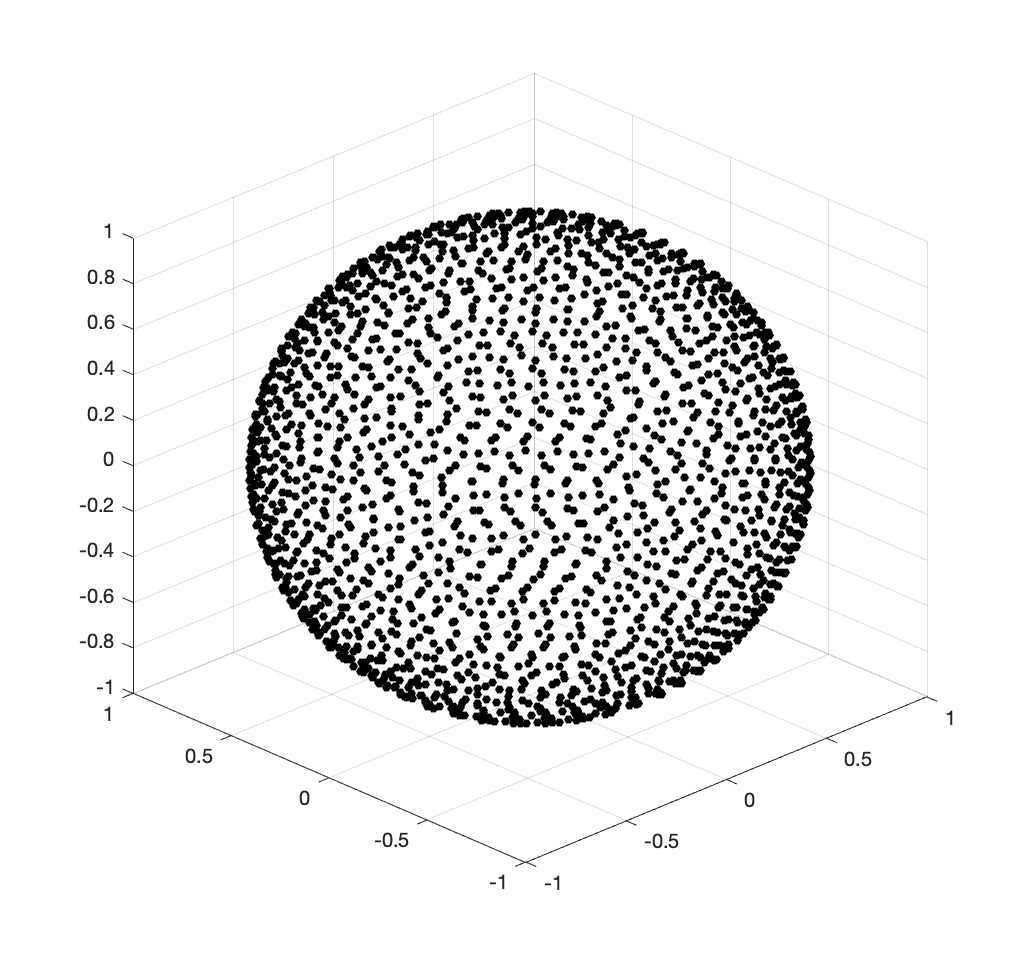}\label{fig:cloud}}
\subfigure[]{\includegraphics[width=0.45\textwidth]{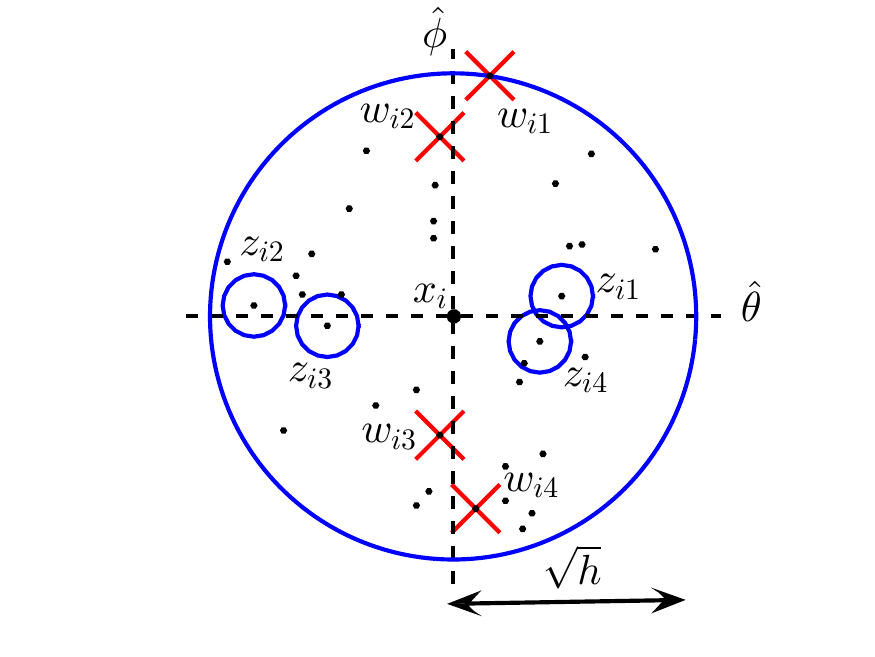}\label{fig:stencil}}
\caption{\subref{fig:cloud}~A point cloud discretizing the sphere ($N=2006$) and \subref{fig:stencil}~the stencil used to discretize \eqref{eq:poisson2d} at the point $x_i$.}%
\label{fig:discSphere}%
\end{figure}

At each point $x_i\in\G^h$, we project a $\sqrt{h}$ neighborhood onto the local tangent plane via geodesic normal coordinates as described in subsections~\ref{sec:normalcoords}-\ref{sec:disc}.  We then need to design a monotone discretization $L^h(x_i,u(x_i)-u(\cdot))$ of the two-dimensional Laplacian at the point $x_i$.  Since we do not have a structured grid in general, we utilize a meshfree finite difference approximation.  In particular, we let $(\hat{\theta},\hat{\phi})$ be our local orthogonal coordinates.  We then choose the discretization points $z_{ij},w_{ij}$, $j=1, \ldots, 4$ that lie in the $jth$ quadrant and are best aligned with $\hat{\theta}$ and $\hat{\phi}$ respectively.  See Figure~\ref{fig:stencil}.  We use the consistency and monotonicity conditions to explicitly compute coefficients $a_{ij},b_{ij} \geq 0$ and define a discrete surface Laplacian of the form
\bq\label{eq:dsiclap}
-\Delta^h_{\mathbb{S}^2}u^h_i = \sum\limits_{j=1}^4a_{ij}(u^h(x_i)-u^h(z_{ij})) + \sum\limits_{j=1}^4b_{ij}(u^h(x_i)-u^h(w_{ij})).
\eq
This yields a formal truncation error of $\bO(\sqrt{h})$; see~\cite{FroeseMeshfreeEigs} for details.

The approach analyzed in section~\ref{sec:convergence} leads us to solve the following system:
\bq\label{eq:disc2d}
\begin{cases}
-\Delta^h_{\mathbb{S}^2}u^h_i + \sqrt{h}u^h_i = f(x_i), & d_{\mathbb{S}^2}(x_i,x_0)>\mathcal{O}(h^{1/6})\\
u^h(x_i) = 0, & d_{\mathbb{S}^2}(x_i,x_0)\leq \mathcal{O}(h^{1/6})
\end{cases}
\eq

Since the formal truncation error of the scheme is $\bO(h^{1/2})$ and the dimension of the manifold is $d=2$, Theorem~\ref{thm:mainconvergence} predicts that the maximum error should scale like $\bO(h^{1/6})$.  This is precisely what we observe in our computations; see Figure~\ref{fig:uerr_2d}.

\begin{figure}%
\centering
\subfigure[]{\includegraphics[width=0.45\textwidth]{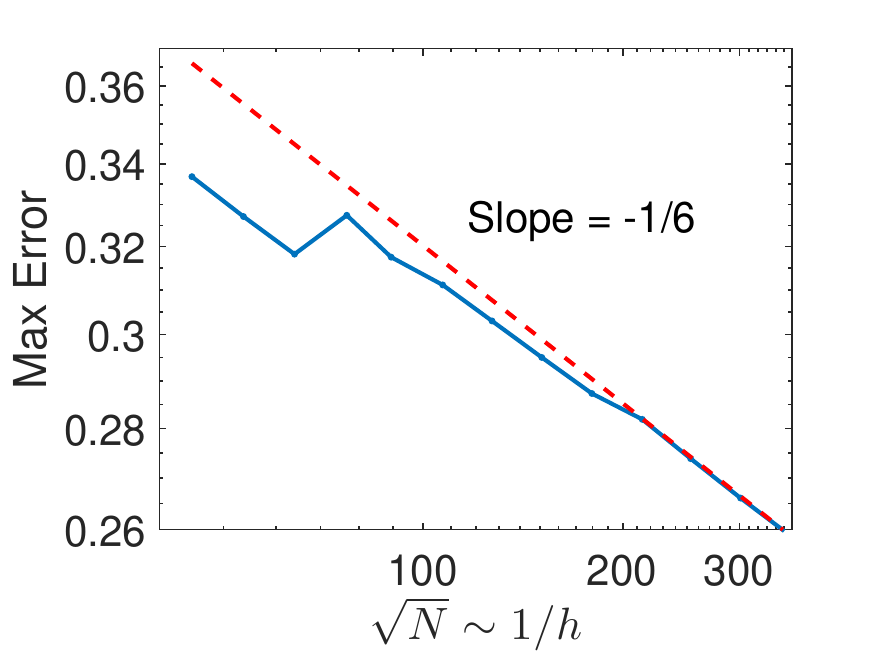}\label{fig:uerr_2d}}
\subfigure[]{\includegraphics[width=0.45\textwidth]{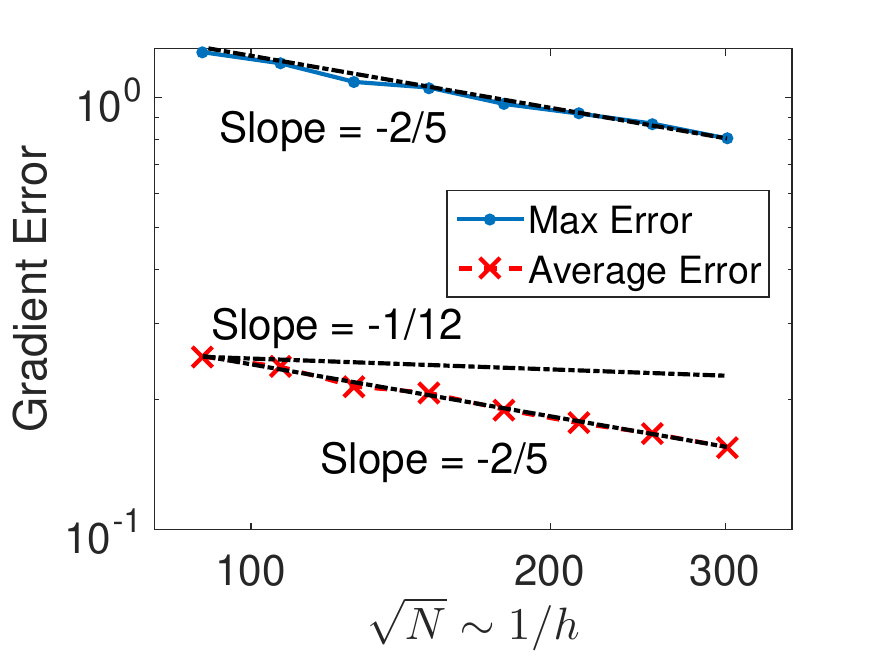}\label{fig:uderiverr_2d}}
\caption{Error in the \subref{fig:uerr_2d}~solution and \subref{fig:uderiverr_2d}~gradient for the solution of Poisson's equation on $\mathbb{S}^2$.}%
\label{fig:poisson2d}%
\end{figure}

Given the low (sublinear) accuracy of these computations, it would be natural to expect that we cannot easily recover information about the surface gradient.  In fact, given that the scheme~\eqref{eq:disc2d} sets $u^h=0$ to be constant on a cap around the origin, it would appear that we should expect a $\bO(1)$ error in any traditional finite difference approximation to the solution gradient.  However, approximation of the gradient is possible using the wider stencil approach discussed in section~\ref{sec:mapping}.  Given a stencil width $r$, we consider the following first-order ($\beta=1$) approximation of the solution gradient at the point $x_i\in\G^h$:
\bq\label{eq:gradapprox}\begin{split}
&j^* = \argmax\limits_j \left\{ \frac{u^h(x_j)-u^h(x_i)}{d_{\mathbb{S}^2}(x_i,x_j)} \mid x_j \in \G^h, 0.5r<d_{\mathbb{S}^2}(x_i,x_j)<r\right\}\\
&\nabla^h_{\mathbb{S}^2}u^h(x_i) = \frac{u^h(x_{j^*})-u^h(x_i)}{d_{\mathbb{S}^2}(x_i,x_{j^*})} \frac{z_{j^*}-z_i}{\abs{z_{j^*}-z_i}}.
\end{split}
\eq
We recall that $z_j$ denotes the projection onto the local tangent plane $\mathcal{T}_{x_i}$ via normal coordinates.

Following Corollary~\ref{cor:errorGrad}, we choose $r = \bO(h^{1/12})$.  Despite the low accuracy of $u^h$, this wider stencil approach successfully approximates the solution surface gradient (Figure~\ref{fig:uderiverr_2d}).  In fact, we observe superconvergence, with an observed error of $\bO(h^{2/5})$ that is significantly better than the $\bO(h^{1/12})$ error bound predicted by Corollary~\ref{cor:errorGrad}.  Indeed, we actually observe a better convergence rate in the gradient than in the approximate solution that was used to estimate the gradient.  Moreover, while the error is artificially large (though still converging to zero) in the cap about the origin, we observe much lower errors throughout most of the domain.  
}

\section{Conclusion}\label{sec:conclusion}
In this manuscript, we studied convergence rates of monotone finite difference approximations for uniformly elliptic PDEs on compact manifolds.  When applied to the Dirichlet problem, solutions of monotone finite difference schemes are expected to converge with an error proportional to their formal consistency error.  We demonstrated empirically that on manifolds without boundary, convergence rates can be lower than the formal consistency error.

We then derived explicit error bounds {for a class of monotone schemes} by carefully constructing barrier functions and exploiting the fact that monotone and proper schemes have a discrete comparison principle. The barrier functions solved a linear elliptic PDE in divergence form with a right-hand side proportional to the formal consistency error of the scheme in the majority of the domain.  However, because of the need to satisfy an additional solvability condition, the right-hand side was permitted to become larger in a small cap on the manifold. This resulted in a barrier function that was asymptotically larger than the formal consistency error.  Because the scaling of the volume of the small cap was dependent on dimension, we found that specific convergence rates depend on the dimension of the underlying manifold. In particular, the reduction in accuracy becomes worse as the dimension increases.

Next, we demonstrated that knowledge of convergence rates can be used to design convergent approximations of the solution gradient through the use of wide finite difference stencils.  We described a family of discrete gradients, with the optimal convergence rate in the gradient bounded by the $L^\infty$ convergence rate of the discrete solution.

Further work will involve utilizing convergence rates for linear elliptic PDEs to prove error bounds for the solutions of fully nonlinear elliptic PDEs. This would apply, for example, to PDEs arising from solving the Optimal Transport problem on the sphere, which is of particular interest due to its application to optical design problems~\cite{Wang_Reflector} and mesh generation~\cite{Weller_OTonSphere}.  The results of this article also highlight the ongoing need to design higher-order numerical methods for elliptic PDEs {in order to compensate the reduction in accuracy that can occur on manifolds without boundary}. {Finally, computational results lead to intriguing questions regarding superconvergence.  In particular, ongoing work will investigate additional conditions (beyond consistency and monotonicity) that would lead to improved error bounds for the computed solution and/or solution gradient.}

\bibliographystyle{plain}
\bibliography{OTonSphere4}

\end{document}